\documentclass[11pt]{amsart}
\usepackage{amsmath}
\usepackage{amssymb}
\usepackage{amsfonts}
\usepackage{mathrsfs}
\usepackage{graphicx}
\usepackage{epsfig}

\theoremstyle{plain}

\newtheorem{prop}{Proposition}[section]

\newtheorem{coro}[prop]{Corollary}
\newtheorem{lemm}[prop]{Lemma}
\newtheorem{theoalph}{Theorem}

\theoremstyle{definition}

\theoremstyle{remark}
\newtheorem{rema}[prop]{Remark}

\newtheoremstyle{citing}
  {3pt}
  {3pt}
  {\itshape}
  {}
  {\bfseries}
  {.}
  {.5em}
  {\thmnote{#3}}

\theoremstyle{citing}
\newtheorem*{generic}{}

\newcommand{\partn}[1]{{\smallskip \noindent \textbf{#1.}}}

\numberwithin{equation}{section}       

\newcommand{\C}{\mathbb{C}}

\newcommand{\R}{\mathbb{R}}

\newcommand{\cW}{\mathcal{W}}
\newcommand{\cX}{\mathcal{X}}

\newcommand{\sA}{\mathscr{A}}

\newcommand{\sF}{\mathscr{F}}
\newcommand{\sG}{\mathscr{G}}

\newcommand{\sK}{\mathscr{K}}
\newcommand{\sL}{\mathscr{L}}
\newcommand{\sM}{\mathscr{M}}
\newcommand{\sN}{\mathscr{N}}

\newcommand{\sU}{\mathscr{U}}

\newcommand{\sX}{\mathscr{X}}
\newcommand{\sY}{\mathscr{Y}}

\newcommand{\hpsi}{\widehat{\psi}}

\newcommand{\teta}{\widetilde{\teta}}

\newcommand{\tsigma}{\widetilde{\tsigma}}
\newcommand{\tvarsigma}{\widetilde{\tvarsigma}}

\newcommand{\la}{\lambda}
\newcommand{\ov}{\overline}

\renewcommand{\=}{ : = }

\DeclareMathOperator{\diam}{diam}
\DeclareMathOperator{\Jac}{Jac}

\DeclareMathOperator{\dist}{dist}

\DeclareMathOperator{\Crit}{Crit}

\newcommand{\CC}{\overline{\C}}
\newcommand{\RR}{\overline{\R}}

\newcommand{\map}{T}

\newcommand{\hsL}{\widehat{\sL}}
\newcommand{\one}{\pmb{1}}
\newcommand{\distancia}{d}
\newcommand{\mapa}{T}
\newcommand{\ExpShrink}{ESC}

\newcommand{\TCEC}{TCE condition}
\DeclareMathOperator{\ES}{ESC}

\DeclareMathOperator{\Per}{Per}

\newcommand{\smjulia}{\widetilde{\sM}(J(\map))}
\newcommand{\mjulia}{\sM(J(\map))}

\newcommand{\imjulia}{\sM(J(\map), \map)}

\newcommand{\average}{W}
\newcommand{\pressure}{P}

\newcommand{\spc}{X} 
\newcommand{\smspc}{\widetilde{\sM}(\spc)} 
\newcommand{\mspc}{\sM(\spc)} 
\newcommand{\imspc}{\sM(\spc, \mapa)} 
\newcommand{\rpot}{\varphi} 
\newcommand{\gpot}{\psi} 
\newcommand{\hgpot}{\hpsi} 
\newcommand{\rrate}{I^{\rpot}} 
\newcommand{\ldpm}{\Omega} 
\newcommand{\ldpmb}{\Sigma} 

\begin{document}

\title[Large deviation principles for TCE rational maps]{Large deviation principles for non-uniformly hyperbolic rational maps}
\author[H. Comman]{Henri Comman$^\dag$}
\author[J. Rivera-Letelier]{Juan Rivera-Letelier$^\ddag$}
\thanks{$\dag$ Partially supported by FONDECYT grant 1070045. Gratefully acknowledges Universidad Cat{\'o}lica del Norte for hospitality.}
\thanks{$\ddag$ Partially supported by Research Network on Low Dimensional Systems, PBCT/CONICYT, Chile. Gratefully acknowledges Universidad de Santiago de Chile for hospitality.}
\address{$\dag$ Henri Comman, Institute  of Mathematics, Pontifical Catholic University of Valparaiso, Chile}
\email{henri.comman@ucv.cl}
\address{$\ddag$ Juan Rivera-Letelier, Facultad de Matem{\'a}ticas, Campus San Joaqu{\'\i}n, P. Universidad Cat{\'o}lica de Chile, Avenida Vicu{\~n}a Mackenna~4860, Santiago, Chile}
\email{riveraletelier@mat.puc.cl}


\begin{abstract}
We show some level-2 large deviation principles for rational maps satisfying a strong form of non-uniform hyperbolicity, called ``Topological Collet-Eckmann''.
More precisely,  we prove a  large deviation principle for the distribution of iterated preimages, periodic points, and Birkhoff averages.
For this purpose we show that  each H{\"o}lder continuous potential admits a unique equilibrium state, and that the pressure function can be characterized in terms of iterated preimages, periodic points, and Birkhoff averages.
Then we use a variant of a general result of Kifer.
\end{abstract}

\maketitle

\section{Introduction}

This paper is devoted to the study of (level-2) large deviation principles for complex rational maps of degree at least two, viewed as dynamical systems acting on the Riemann sphere.
Our results apply to rational maps satisfying a strong form of non-uniform hyperbolicity condition, called ``Topological Collet-Eckmann'' (TCE).
Although the \TCEC{} is very strong, the set of rational maps that satisfy it, but that are not uniformly hyperbolic, has positive Lebesgue measure in the space of rational maps of a given degree~\cite{Aspthesis}; see also~\cite{Ree86,GraSwi00,Smi00,DujFav08} for related results.
The \TCEC{} is also interesting because it can be formulated in several equivalent ways~\cite{PrzRivSmi03}.

The first key observation is that for a rational map satisfying the \TCEC{} every H\"older continuous potential has a unique equilibrium state.
This allows us to apply (a variant of) a general result of Kifer~\cite[Theorem~3.4]{Kif90} to obtain level\nobreakdash-2 large deviation principles for sequences of measures associated to periodic points, iterated preimages, and Birkhoff averages.

We now proceed to describe our results in more detail.

\subsection{Equilibrium states for TCE rational maps}\label{ss:equilibria}
Let~$\map$ be a complex rational map  of degree at least two, viewed as a dynamical system acting on the Riemann sphere~$\CC$.
We denote by~$J(\map)$ its Julia set and by $\imjulia$ the space of invariant probability measures supported by ~$J(\map)$, endowed with the weak$^*$ topology.
For each $\mu \in \imjulia$ we denote by~$h_\mu(\map)$ the \textit{measure-theoretic entropy of~$\mu$}.
Given a H\"older continuous function~$\rpot : ~J(\map) \to \R$, a probability measure~$\mu_0 \in \imjulia$ is called an ~\textit{equilibrium state of~$\map$ for the
potential~$\rpot$}, if the supremum
\begin{equation}\label{e:variational pressure}
\pressure(\map, \rpot)
\=\sup \left\{ h_\mu(T) + \int \rpot d\mu : \mu \in \imjulia \right\},
\end{equation}
is attained at $\mu = \mu_0$.

The \TCEC{} was originally formulated in topological terms.
It is equivalent to the following strong form of Pesin's non-uniform hyperbolicity condition: There is a constant $\chi > 0$ such that for each~$\mu \in \imjulia$ the \textit{Lyapunov exponent $\int \log |\map'| d\mu$ of~$\mu$} is greater than or equal to~$\chi$.
See~\cite{PrzRivSmi03} for the original formulation of the \TCEC, and  several others  equivalent formulations.
For other results concerning equilibrium states of rational maps see~\cite{MakSmi03,PrzRiv0806,StrUrb03} and references therein.

The following result is fundamental in what follows.
\begin{theoalph}\label{t:equilibria}
Let~$\map$ be a rational map satisfying the \TCEC.
Then for every H\"older continuous function $\rpot : J(\map) \to \R$ there is a unique equilibrium state of~$\map$ for the potential~$\rpot$.
\end{theoalph}
We obtain this theorem as a simple consequence of~\cite[Theorem~8]{Dob0804}.
In Appendix~\ref{s:Ruelle-Perron-Frobenius} we give a reasonably self contained proof of this result, as a consequence of a Ruelle-Perron-Frobenius type theorem (Theorem~\ref{t:Ruelle-Perron-Frobenius}).
When the potential~$\rpot$ satisfies $\sup_{J(\map)} \rpot < \pressure(\map, \rpot)$, these results were shown for a general rational map~$\map$ in~\cite{DenUrb91b,Prz90,DenPrzUrb96}.
The fact that Theorem~\ref{t:equilibria} holds for \textit{every} H\"older continuous potential is  crucial to obtain the large deviation principles that we proceed to describe.

\subsection{Level-2 large deviations principles for TCE rational maps}\label{ss:LDP TCE}
Let~$\mjulia$ be the space of Borel probability measures on~$J(\map)$ endowed with the weak$^*$ topology, and let  $I : \sM(J(\map)) \to [0, + \infty]$ be a lower semi-continuous function. Recall that a sequence~$(\ldpm_n)_{n \ge 1}$ of Borel probability measures on~$\sM(J(\map))$ is said to satisfy a \textit{large deviation principle with rate function~$I$}, if for every closed subset $\sF$ of $\sM(J(\map))$ we have
\[
\limsup_{n \to + \infty} \frac{1}{n} \log \ldpm_n (\sF)
\le - \inf_{\sF} I,
\]
and if for every open subset~$\sG$ of~$\sM(J(\map))$ we have,
\[
 \liminf_{n \to + \infty} \frac{1}{n} \log \ldpm_n (\sG)
\ge - \inf_{\sG} I.
\]
The function~$I$ is uniquely characterized by this property, see~\S\ref{s:preliminaries} for background and further properties.

\begin{theoalph}\label{t:LDP TCE}
Let~$\map$ be a rational map satisfying the \TCEC, let~$\rpot : J(\map) \to \R$ be a H\"older continuous function, and let~$\mu_{\rpot}$ be the unique equilibrium state of~$\map$ for the potential~$\rpot$.
For each integer~$n \ge 1$ let $\average_n : J(\map) \to \mjulia$ be the continuous function defined by
$$ \average_n(x) \= \tfrac{1}{n} \left( \delta_x + \delta_{\map(x)} + \cdots + \delta_{\map^{n - 1}(x)} \right), $$
and let~$S_n(\rpot) : J(\map) \to \R$ be defined by
$$ S_n(\rpot)(x) \= n \int \rpot d\average_n(x) = \rpot(x) + \rpot \circ \map(x) + \cdots + \rpot \circ \map^{n - 1}(x). $$
Given an integer~$n \ge 1$ consider the following Borel probability measures on~$\mjulia$.
\begin{description}
\item[Periodic points]
Letting~$\Per_n \= \{ p \in J(\map) \mid \map^n(p) = p \}$, put
$$ \ldpm_n \= \sum_{p \in \Per_n} \frac{\exp(S_n(\rpot)(p))}{\sum_{p' \in \Per_n} \exp(S_n(\rpot)(p'))} \delta_{\average_n(p)}. $$
\item[Iterated preimages]
Given $x_0 \in J(\map)$, put
$$ \ldpm_n(x_0) \= \sum_{x \in \map^{- n}(x_0)} \frac{\exp(S_n(\rpot)(x))}{\sum_{y \in \map^{-n}(x_0)} \exp(S_n(\rpot)(y))} \delta_{\average_n(x)}. $$
\item[Birkhoff averages]
$ \ldpmb_n \= \average_n[\mu_\rpot]$ (\textit{i.e.}, the image measure of~$\mu_\varphi$ by~$\average_n$).
\end{description}
Then each of the sequences $( \ldpm_n )_{n \ge 1}$, $( \ldpm_n(x_0) )_{n \ge 1}$ and $(\ldpmb_n)_{n \ge 1}$ converges to~$\delta_{\mu_{\rpot}}$ in the weak$^*$ topology, and satisfies a large deviation principle in~$\mjulia$ with rate function $\rrate : \mjulia \to [0,+\infty]$ given by
\begin{equation}\label{t:LDP TCE-eq1}
\rrate(\mu)
=
\begin{cases}
\pressure (\mapa, \rpot) - \int \rpot d\mu - h_{\mu}(\mapa) & \text{if } \mu\in\imjulia;
\\
+\infty & \text{if } \mu \in \mjulia \setminus \imjulia.
\end{cases}
\end{equation}
Furthermore, for each convex open subset~$\sG$ of~$\mjulia$ containing some invariant measure  we have
$\inf_{\sG} \rrate=\inf_{\overline{\sG}} \rrate$, and
\begin{multline}\label{t:LDP TCE-eq2}
\lim_{n \to + \infty} \frac{1}{n} \log \ldpm_n(\sG)
=
\lim_{n \to + \infty} \frac{1}{n} \log \ldpm_n(x_0)(\sG)
=
\lim_{n \to + \infty} \frac{1}{n} \log \ldpmb_n(\sG)=\inf_{\sG} \rrate,
\end{multline}
and the above expression remains true replacing  $\sG$ by $\overline{\sG}$.
\end{theoalph}

In order to illustrate Theorem~\ref{t:LDP TCE} we state a couple of corollaries.
\begin{coro}\label{c:contracted LDP}
Let  $\gpot : J(\map) \to \R$ be a continuous function, and let  $\widehat{\psi}:\mjulia \to \R$ be defined by
$\widehat{\psi}(\mu)=\int\psi d\mu$.
With the notations of Theorem~\ref{t:LDP TCE}, each of the sequences of image measures $(\widehat{\psi}[\Omega_n])_{n \ge 1}$,  $(\widehat{\psi}[\Omega_n(x_0)])_{n \ge 1}$, $(\widehat{\psi}[\ldpmb_n])_{n \ge 1}$ satisfies a large deviation principle in $\mathbb{R}$ with rate function
\[x\mapsto\inf\left\{ I^\varphi(\mu): \mu\in\mjulia, \int\psi d\mu=x \right\}.\]
Furthermore, when $\psi$ is normalized so that $\int \gpot d\mu_{\rpot} = 0$, for each $\varepsilon > 0$ small enough we have
\begin{equation*}
\begin{split}
& \quad \lim_{n \to + \infty} \frac{1}{n} \log \left( \frac{\sum_{{p \in \Per_n, \tfrac{1}{n} |S_n(\gpot)(p)|  > \varepsilon}} \exp(S_n(\rpot)(p))}{\sum_{p' \in \Per_n} \exp(S_n(\rpot)(p'))}
 \right)
\\ & =
\lim_{n \to + \infty} \frac{1}{n} \log \left( \frac{\sum_{{x \in \map^{-n}(x_0), \tfrac{1}{n} |S_n(\gpot)(x)|  > \varepsilon}} \exp(S_n(\rpot)(x))}{\sum_{y \in \map^{-n}(x_0)} \exp(S_n(\rpot)(y))}
 \right)
\\ & =
\lim_{n \to + \infty} \frac{1}{n} \log \mu_{\rpot} \left\{ x \in J(\map) : \tfrac{1}{n} |S_n(\gpot)(x)|  > \varepsilon \right\}
\end{split}
\end{equation*}
\begin{equation}\label{c:contracted LDP-eq2}
=
-\inf \left\{ \pressure(\map, \rpot) - \int \rpot d\mu - h_{\mu}(\map): \mu \in \imjulia,  \left| \int \gpot d\mu \right|  > \varepsilon \right\},
\end{equation}
and the above limits are strictly negative (possibly infinite).
\end{coro}

\begin{coro}\label{c:entropy expression}
With the notations of Theorem~\ref{t:LDP TCE}, for each $\mu \in \imjulia$ and each convex local basis~$\sG_\mu$ at~$\mu$, we have
\begin{equation*}
\begin{split}
h_\mu(\map) & + \int \rpot d\mu \\
& =
\inf \left\{
\lim_{n \to + \infty} \frac{1}{n}\log\sum_{p \in \Per_n, \average_n(p)\in\sG} \exp(S_n(\rpot)(p)
: \sG\in\sG_\mu \right\},
\\ & =
\inf \left\{
\lim_{n \to + \infty} \frac{1}{n}\log\sum_{x \in \map^{-n}(x_0), \average_n(x)\in\sG} \exp(S_n(\rpot)(x)
: \sG\in\sG_\mu \right\},
\\ & =
P(\map,\varphi) + \inf \left\{
\lim_{n \to + \infty} \frac{1}{n}\log\mu_\varphi\{x\in J(\map):W_n(x)\in\sG\}
: \sG\in\sG_\mu \right\}.
\end{split}
\end{equation*}
\end{coro}

Theorem~\ref{t:LDP TCE} was obtained recently by the first named author in the case where~$\map$ is uniformly hyperbolic~\cite[Theorem~2]{Com09}.\footnote{Taking~$\map$ uniformly hyperbolic in Theorem~\ref{t:LDP TCE} does not permit to recover all the cases treated in~\cite[Theorem 2]{Com09}; this comes from the fact that in this last paper, the potential~$\varphi$ need not have a unique equilibrium state, as it is required in Theorem \ref{t:LDP TCE} (see~\cite[Example 4.1]{Com09}).}
For the same class of maps, the case of Birkhoff averages and $\varphi=0$ was obtained earlier by Lopes~\cite{Lop90}, and the upper-bounds in the case of periodic points were proved by Pollicott and Sridharan in~\cite{PolSri07}.

The Birkhoff averages case of Theorem~\ref{t:LDP TCE} was obtained by Grigull when~$\map$ is a parabolic rational map and when the potential~$\rpot$ satisfies $\sup_{\CC} \rpot < P(\map, \rpot)$~\cite[Theorem~1]{Grithesis}. See also the survey paper of Denker \cite{Den96}.

The large deviation upper bounds in the case of iterated preimages have been proved  by Pollicott and Sharp in~\cite{PolSha96} for an arbitrary rational map~$\map$, when the potential~$\rpot$ satisfies $\sup_{J(\map)} \rpot < \pressure(\map, \rpot)$.
An alternative proof of this result can be obtained using a general result on upper bounds, see~\cite[Remark~2 and Theorem~4]{Com09} and \cite[Theorem~4.5.3]{DemZei98}. See also~\cite{PolShaYur98} for the upper bounds in the case of interval maps with indifferent periodic points.

Using the contraction principle it is possible to derive from~Theorem~\ref{t:LDP TCE} a level\nobreakdash-1 large deviation principle in~$\R$ for each continuous potential, as in Corollary~\ref{c:contracted LDP}\footnote{See also the inducing scheme approach (of level-1 large deviations)  given recently by Melbourne and Nicol~\cite{MelNic08}, and Rey-Bellet and Young~\cite{ReyYou08}.}.
However, this simple trick does not work  with the geometric potential $- \log |\map'|$ by the lack of continuity of  the evaluation map $\mu\mapsto\int\log |\map'|d\mu$ when there is a critical point in the Julia set.
The techniques needed in order to get (even partial) level\nobreakdash-1 large deviations with the potential~$- \log |\map'|$ are different from those used here, and we shall not tackle them in this paper.
We refer here to results where large deviation bounds are proved only for some subsets of the real line, like for example those obtained by Keller
and Nowicki~\cite[Theorem~1.2 and Theorem~1.3]{KelNow92} in the case of unimodal maps satisfying the Collet-Eckmann condition, or~\cite[Corollary~B.4]{PrzRiv0806} and \cite{FuXia07} in the case of rational maps\footnote{See~\cite{Dinsib0810} for a weak form of upper bounds in the higher dimensional setting.}.

\subsection{Abstract result on level\nobreakdash-2 large deviations principles}
\label{ss:general LDP}
Theorem~\ref{t:LDP TCE} is obtained as a particular case of the following
variant of  Kifer's result~\cite[Theorem~3.4]{Kif90}.
See Appendix~\ref{s:flows} for an extension to more general dynamical systems and nets in place of sequences.
\begin{theoalph}\label{t:LDP}
Let~$\spc$ be a compact metrizable topological space, and let~$\mapa : \spc \to \spc$ be a continuous map  such that the measure-theoretic entropy of~$\mapa$, as a function defined on~$\imspc$, is finite and upper semi-continuous.
Fix $\rpot \in C(X)$, and let ~$\cW$ be a dense vector subspace of $C(X)$ such that for each $\gpot \in \cW$ there is a unique equilibrium state of~$\mapa$ for the potential~$\varphi+\psi$.
 Let $\rrate : \mspc \to [0, + \infty]$ be the function defined by
$$
\rrate(\mu)
=
\left\{
  \begin{array}{ll}
\pressure (\mapa, \rpot) - \int \rpot d\mu - h_{\mu}(\mapa) & \text{if } \mu\in\imspc; \\
+\infty & \text{if } \mu \in \mspc \setminus \imspc.
  \end{array}
\right.
$$
Then every sequence~$( \ldpm_n )_{n \ge 1}$  of Borel probability measures on~$\mspc$ such that for every~$\psi\in\cW$,
\begin{equation}\label{e:functional equality}
\lim_{n \to + \infty} \frac{1}{n} \log \int_{\mspc} \exp \left( n \int \gpot d \mu \right) d \ldpm_n(\mu)
=
\pressure (\mapa, \rpot + \gpot) - \pressure ( \mapa, \rpot),
\end{equation}
satisfies  a large deviation principle   with rate   function $\rrate$, and it converges in the weak$^*$ topology to the Dirac mass supported on the unique equilibrium state of~$\mapa$ for the potential~$\rpot$. Furthermore, for each convex and open subset~$\sG$ of~$\mspc$ containing some invariant measure, we have
\[
\lim_{n \to + \infty} \frac{1}{n} \log \ldpm_n(\sG)
=
\lim_{n \to + \infty} \frac{1}{n}\log\ldpm_n(\overline{\sG})
=
-\inf_{\sG} \rrate
=
-\inf_{\overline{\sG}} \rrate.
\]
\end{theoalph}

The method we use to prove Theorem~\ref{t:LDP} is in line with the general functional approach of large deviations in probability theory.
This approach seems to have been initiated by  Sievers in \cite{Sie69} and then by Plachky and Steinebach in \cite{Pla71,PlaSte75} in order  to generalize to sequences of   dependent random variables  the large deviation principle proved by Cramer (in a special case)  and Chernoff (in the general case) for the laws  of empirical means of independent and identically distributed random variables in~$\R$ \cite{Cra38,Che52}.  The result was extended to $\mathbb{R}^d$-valued random variables in \cite{Gar77}, and then  refined by Ellis  leading to the well-known G\"{a}rtner-Ellis theorem in~\cite{Ell84}, that was later generalized by Baldi in~\cite{Bal88} to  real  topological vector spaces.

For  the case of dynamical systems,   Takahashi in \cite{Tak84,Tak87} studied the  large deviation functional associated to the  distributions of Birkhoff averages with respect to some (not necessarily invariant) measure.
  Then, in a very general setting, Kifer   gave  sufficient conditions in order to get the  large deviation  principle  with convex rate function,  for empirical measures \cite[Theorem 2.1]{Kif90}.
This result can be seen as  a purely theoretic large deviation one, in the sense that the hypotheses do not depend on a system under which the  empirical  measures could evolve (see Remark~\ref{remark-proof-LDP-Thm}).
This allowed Kifer to derive more specific results for dynamical systems; the first one
 concerns  the distribution of these empirical  measures with respect to  some reference measure, like in  the third case  of Theorem~\ref{t:LDP TCE} \cite[Theorem 3.4]{Kif90} (see Appendix~\ref{s:flows}); the second  one deals with the case where  these measures are governed by a Markov process \cite[Theorem 4.1]{Kif90}.
Recently, the first named author gave another type of sufficient condition in order to get a large deviation principle with the same rate function~\cite[Theorem 4]{Com09}.

In all the above results, the first basic assumption relates the pressure to the  large deviation functional associated to the sequence or net of measures (see \S\ref{s:preliminaries}).
Roughly speaking,  it is required that
 the (translated) pressure functional coincides with the large deviation functional; rigorously, this means that~\eqref{e:functional equality} holds for all $\gpot \in C(X)$ (or equivalently, for all $\gpot$ in a dense subset of $C(X)$).
 It turns out that the existence of the limit in the left hand side of~\eqref{e:functional equality} is also necessary in order to have  the large deviation principle, and the fact that it coincides with the pressure is necessary in order to have the rate function of  Theorem~\ref{t:LDP TCE} (see Remark~\ref{remark-equivalence-of-condition-in-theo-C}).

  The second basic assumption is in fact a condition on the large deviation functional in disguise; we refer the reader to  Remark~\ref{remark-proof-LDP-Thm} and Appendix~\ref{s:flows} in the case of  Kifer's theorem.
In the case of~\cite{Com09}, it is required  that every invariant measure can be approximated in the weak$^*$ topology, and in entropy, by measures which are unique equilibrium states for some potentials; when (\ref{e:functional equality}) holds for all $\psi\in C(X)$, this turns out to be the usual Baldi's condition in large deviation theory~\cite{Bal88}.

  We can summarize the functional approach by saying it consists to look  for sufficient conditions on the large deviation functional implying  the large deviation principle.  The rate function~\eqref{t:LDP TCE-eq1} is then a natural candidate when the first above mentioned basic assumption  holds, since in this case it is the only possible convex rate function (namely,  the Legendre-Fenchel transform of the restriction of the large deviation functional to the topological dual of the space of finite signed Borel measures on $X$ (\textit{i.e.} $C(X)$);  see \cite{DemZei98} and  \cite{Com09} in connection with Remark~\ref{remark-proof-LDP-Thm}).

\subsection{Organization}
\label{ss:organization}
After some preliminaries in~\S\ref{s:preliminaries}, we give the proof of Theorem~\ref{t:LDP} in~\S\ref{s:LDP}.
In Appendix~\ref{s:flows} we use this result to give another variant of Kifer's result for semi-flows~\cite[Theorem~3.4]{Kif90}, that we state as Theorem~\ref{Kifer-theo}.

We start~\S\ref{s:LDP TCE} by deriving the proof of Theorem~\ref{t:equilibria} from~\cite[Theorem~8]{Dob0804} in~\S\ref{ss:transfer operator}.
Then we obtain Theorem~\ref{t:LDP TCE} and its corollaries in~\S\ref{ss:proof of LDP TCE}, from Theorem~\ref{t:equilibria} and Theorem~\ref{t:LDP}, using several characterizations of the pressure given in~\S\ref{ss:characterizations pressure}.

In Appendix~\ref{s:Ruelle-Perron-Frobenius} we give a reasonably self contained proof of Theorem~\ref{t:equilibria} as a consequence of a Ruelle-Perron-Frobenius type theorem (Theorem~\ref{t:Ruelle-Perron-Frobenius}).

\subsection{Acknowledgements} We thank Godofredo Iommi for a useful remark
concerning Theorem~\ref{t:equilibria}.

\section{Preliminaries}\label{s:preliminaries}

\subsection{Notation}
\label{ss:notation}
We denote by $\RR = \R \cup \{- \infty, + \infty \}$ the extended real line.
We denote by~$\dist$ the spherical metric on~$\CC$.
Given a subset~$E$ of~$\CC$ we denote by~$\one_E$ the indicator function of~$E$.
We will denote~$\one_{\CC}$ simply by~$\one$.

\subsection{Measure spaces}
Given a compact metric  space~$\spc$, we denote by~$C(\spc)$ the space of continuous functions defined on~$X$ taking images in~$\R$, endowed with the uniform topology.
We identify the dual of~$C(\spc)$ with the  space~$\smspc$ of finite signed Borel measures on~$\spc$ endowed with the weak$^*$ topology \cite[\S{}IV.6, Theorem~3]{DunSch88a}.
We denote by $\mspc\subset\smspc$ the space of Borel probability measures on $X$, and recall that  $\mspc$ is compact \cite[\S{}V.4, Theorem 2]{DunSch88a} and metrizable \cite[\S{}V.5, Theorem~1]{DunSch88a}.
If~$\mapa : \spc \to \spc$ is a continuous map, then we denote by $\imspc$ the  compact subset of $\mspc$ constituted by the measures that are invariant by~$\mapa$.

\subsection{Convex analysis}
\label{ss:convex analysis}
Let~$\cX$ be a locally convex Hausdorff real topological vector space, and let~$\cX^*$ be its topological dual.
The \textit{Legendre-Fenchel} transform  of a function~$f : \cX \to \RR$ is by definition the function $f^* : \cX^* \to \RR$ defined by
$$ f^*(u)
=
\sup \left\{ u(x) - f(x) : x \in \cX \right\}. $$
The \textit{duality theorem} asserts that if~$f$ is convex, lower semi-continuous and takes values in~$(- \infty, + \infty]$, then for each $x \in \cX$ we have
$$ f(x) = \sup \{ u(x) - f^*(u) : u \in \cX^* \}; $$
see for example~\cite[\S{}I, Proposition~4.1]{EkeTem76}.

\subsection{Thermodynamic formalism}
\label{ss:thermodynamic formalism}
The reader may refer to~\cite{Wal82,Rue04} for background in ergodic theory and thermodynamic formalism, and~\cite{PUbook,Zin96} for an introduction in the case of rational maps.

Let~$\spc$ be a compact metric space with metric $d$,  and let~$\mapa : \spc \to \spc$ be a continuous map.
For $\mu \in \imspc$ we will denote by~$h_{\mu}(\mapa)$ the measure-theoretic entropy of~$\mu$.
We now recall the definition of topological pressure through ``$(n, \varepsilon)$\nobreakdash-separated sets'', that will be needed in~\S\ref{ss:characterizations pressure}.
 Denote by $\mapa \times \mapa : X \times X \to X \times X$
the diagonal action defined by $\mapa \times \mapa (x, x') = (\mapa(x), \mapa(x'))$.
Given an integer $n \ge 1$ we denote by $\distancia_n$ the distance
on~$X$ defined by
$$
\distancia_n
=
\max \left\{ \distancia \circ (\mapa \times \mapa)^j : j \in \{ 0, \ldots, n - 1 \} \right\}.
$$
Note that $\distancia_1 = \distancia$.
Given $\varepsilon > 0$ we say that a subset~$\sN$ of~$X$ is $(n, \varepsilon)$\nobreakdash-\textit{separated}, if for each pair of distinct elements~$x$, $x'$ of~$\sN$ we have $\distancia_n(x, x') > \varepsilon$.
For an integer $n \ge 1$ and a continuous function $\rpot : X \to \R$ we put
$$ S_n(\rpot) = \rpot + \rpot \circ \mapa + \cdots + \rpot \circ \mapa^{n - 1}. $$
Then the pressure function is equal to
$$
\pressure(\mapa, \rpot)
=
\lim_{\varepsilon \to 0} \lim_{n \to + \infty} \sup_{\sN} \sum_{y \in \sN} \exp(S_n(\rpot)(y)),
$$
where the supremum is taken over all $(n, \varepsilon)$\nobreakdash-separated subsets~$\sN$ of~$X$.
The fact that the pressure function defined with $(n, \varepsilon)$\nobreakdash-separated sets as above is equal to the supremum in~\eqref{e:variational pressure}, is known as \textit{the variational principle}.
When the topological entropy of~$\mapa$ is finite, the topological pressure viewed as a function defined on~$C(\spc)$, takes finite values and it is Lipschitz continuous~\cite[Theorem~9.7]{Wal82}.

\subsection{Large deviations}\label{prelim-LDP}
We recall  here some basic facts of large deviation theory that  will be used in the sequel.
Since  we will allude to large deviations for nets in place of sequences, and in various types of topological spaces, we state  them in a  general topological setting, and  refer the reader to \cite{DemZei98,Com03,MR2363442,Ell85} for more details.

 Let $(\Omega_\alpha)$ be a net  of
Borel probability measures on a Hausdorff topological space $\sX$, and let $(t_\alpha)$ be a net in $(0,+\infty)$ converging to~$0$.
We say that $(\Omega_\alpha)$ satisfies a \textit{large deviation principle with powers $(t_\alpha)$} if there exists a
 lower semi-continuous function $I:\sX\rightarrow[0,+\infty]$  such
that
\begin{equation}\label{prelim-LDP-eq1}
\limsup_{t_\alpha\rightarrow 0} \  t_\alpha\log\Omega_{\alpha}(\sF)\le- \inf \left\{ I(x) : x \in \sF \right\} \ \ \ \ \ \ \ \textnormal{for all closed $\sF\subset \sX$,}
\end{equation}
and
 \begin{equation}\label{prelim-LDP-eq1.1}
 \liminf_{t_\alpha\rightarrow 0}\  t_\alpha\log\Omega_\alpha(\sG)\ge - \inf \left\{ I(x) : x\in \sG \right\} \ \ \ \ \ \ \ \textnormal{for all open $\sG\subset \sX$.}
 \end{equation}
Such a function~$I$ is then unique when~$\sX$ is regular; it is called the \textit{rate function}, and is given for each $x\in \sX$ and each local basis~$\sG_x$ at~$x$ by
\begin{equation}\label{prelim-LDP-eq2}
\begin{split}
-I(x)
& = \inf \left\{ \liminf_{t_\alpha\rightarrow 0}\  t_\alpha\log\Omega_\alpha(\sG) : \sG\in\sG_x \right\}
\\ & = \inf \left\{ \limsup_{t_\alpha\rightarrow 0}\  t_\alpha\log\Omega_\alpha(\sG) : \sG \in \sG_x \right\}.
\end{split}
\end{equation}
A Borel set
 $\sA\subset \sX$ is called  a \textit{$I$-continuity set} if
\[
\inf \left\{ I(x) : x \in \textnormal{Interior}(\sA) \right\}
=
\inf \left\{ I(x) : x \in \overline{\sA} \right\}.
\]
When \eqref{prelim-LDP-eq1} and ~\eqref{prelim-LDP-eq1.1} hold, then
$\lim_{t_\alpha\rightarrow 0} t_\alpha\log\Omega_\alpha(\sA)$ exists and satisfies
\[\lim_{t_\alpha\rightarrow 0} t_\alpha\log\Omega_\alpha(\sA) = - \inf \left\{ I(x) : x\in \sA\right\} \]
and we can replace~$\sA$ by either its interior or its closure in the above equality.
When only~\eqref{prelim-LDP-eq1} (resp.~\eqref{prelim-LDP-eq1.1}) is satisfied,  we say that the \textit{large deviation upper (resp. lower) bounds} hold with the function~$I$.

The \textit{contraction principle} asserts that when~$(\Omega_\alpha)$ is supported by a compact subset of~$\sX$ and~$(\Omega_\alpha)$  satisfies a large
deviation principle with powers~$(t_\alpha)$ and rate function $I$, then for every Hausdorff topological space~$\sY$ and any continuous map $g:\sX\rightarrow \sY$ the net of image measures~$(g[\Omega_\alpha])$ satisfies a  large
deviation principle with powers~$(t_\alpha)$ and rate function  defined on $\sY$ by
\[y\mapsto\inf \{I(x):x\in \sX,g(x)=y\}.\]

The \textit{large deviation functional} associated to~$(\Omega_\alpha)$ and~$(t_\alpha)$ is the map defined on the set of $[-\infty,+\infty)$-valued Borel
functions~$h$ on~$\sX$ by
\begin{equation}\label{prelim-LDP-eq4}
h\mapsto\limsup_{t_\alpha \to 0}   t_\alpha\log\int \exp\left( h/t_\alpha \right ) d\Omega_\alpha;
\end{equation}
it
 is continuous  with respect to  the uniform metric.
Assume  that $\sX$ is a  Hausdorff real topological vector space, let $\sX^*$ denote its topological dual endowed with the weak$^*$ topology, and let
 $\overline{L}$ be the restriction of the large deviation functional (\ref{prelim-LDP-eq4})  to $\sX^*$; for $u\in\sX^*$ we shall write~$L(u)$ when the limit  exists in (\ref{prelim-LDP-eq4}).
When the net $(\Omega_\alpha)$ is supported by a  compact subset $\sK\subset\sX$, then $\overline{L}$ is a convex lower semi-continuous function.
In the literature, $\overline{L}$ is also known as the ``generalized log-moment generating function'', ``free-energy'', or ``pressure'', depending of the context.

The above notions will be applied   with $\sX=\widetilde{\sM}(X)$ (strictly speaking,~$\sX$ will be  homeomorphic to $\widetilde{\sM}(X)$),  $\sK=\sM(X)$, $\sY=\mathbb{R}$ and $g=\widehat{\psi}$ for some  $\psi\in C(X)$, where  $\widehat{\psi}$ is the evaluation map (\textit{i.e.} $\widehat{\psi}(\mu)=\int\psi d\mu$). Note that if $L(\widehat{\psi})$ exists for all $\psi$ in a  dense subset of $C(X)$, then  $L(\widehat{\psi})$ exists for all $\psi\in C(X)$.
In this context, the large deviation principles in~$\mspc$, or more generally in~$\smspc$, are usually referred to as ``level\nobreakdash-2'', and the ones in~$\mathbb{R}$ (in particular those obtained by contraction) as ``level\nobreakdash-1''.

\section{Proof of Theorem~\ref{t:LDP}}\label{s:LDP}
This section is devoted to the proof of Theorem~\ref{t:LDP}.
It is based on Lemma~\ref{l:LDP} below, which identifies the rate function as a Legendre-Fenchel transform.
Throughout the rest of this section we fix~$\spc, \mapa, \cW, \rpot$  as in the statement of Theorem~\ref{t:LDP}.
Note that the hypothesis of Theorem~\ref{t:LDP}, that the measure-theoretic entropy  is finite and upper semi-continuous, implies that for every~$\psi \in C(\spc)$ the  pressure~$\pressure(\mapa, \psi)$ is finite.

\begin{lemm}\label{l:LDP}
Let $Q_{\rpot} : C(\spc) \to \R$ be the function defined by
$$ Q_\rpot(\gpot) = \pressure( \mapa, \rpot + \gpot) - \pressure(\mapa, \rpot).$$
Then the following properties hold.
\begin{itemize}
\item[1.]
The function~$Q_{\rpot}$ is continuous, convex, and its Legendre-Fenchel transform $Q_\rpot^*$  is given by
$$ Q_{\rpot}^*(\mu)
=
\begin{cases}
\pressure ( \mapa, \rpot) - \int \rpot d \mu - h_{\mu}(\map) & \text{if } \mu\in \imspc; \\
+\infty & \text{if } \mu \in \smspc \setminus \imspc.
\end{cases}
$$
In particular~$Q_{\rpot}^*$ takes images in~$[0, + \infty]$, and it vanishes precisely on the set of equilibrium states of~$\mapa$ for the potential~$\rpot$.
Note furthermore that ${Q_{\rpot}^*}_{\mid_{ \mspc}}= \rrate$.
\item[2.] For each $\psi\in C(X)$, a measure $\mu\in\sM(X,\map)$ is an  equilibrium state of~$\mapa$ for the potential $\varphi+\psi$ if and only if
$Q_\varphi(\psi)=\int\psi d\mu-{Q_{\rpot}^*}(\mu)$.
\end{itemize}
\end{lemm}

\begin{proof}
The  convexity  and the  continuity of
 $Q_{\rpot}$  follow from the same properties of  the pressure function, see~\S\ref{ss:thermodynamic formalism}.
Let~$U : \smspc \to \RR$ be defined by
$$
U(\mu)
=
\begin{cases}
- \int \rpot d\mu  - h_{\mu}(\mapa) & \text{if } \mu \in \imspc;
\\
+\infty & \text{if } \mu \in \smspc \setminus \imspc.
\end{cases}
$$
Since the measure-theoretic entropy of~$\mapa$ is affine, and since it is upper semi-continuous by hypothesis, it follows that the function~$U$ is convex, lower semi-continuous, and that it takes values in $(-\infty,+\infty]$.
By the variational principle, for each~$\gpot \in C(\spc)$ we have
\begin{equation*}
\begin{split}
\pressure (\mapa, \rpot + \gpot)
& =
\sup \left\{ h_\mu (\mapa) + \int \rpot + \gpot d\mu : \mu \in \imspc \right\}
\\ & =
\sup \left\{ \int \gpot d\mu - U(\mu) : \mu \in \smspc \right\}.
\end{split}
\end{equation*}
This shows that the function $\gpot \mapsto \pressure(\mapa, \rpot + \gpot)$ is the Legendre-Fenchel transform of~$U$.
Hence, the duality theorem implies that for each $\mu \in \smspc$ we have,
\begin{equation*}
\begin{split}
U(\mu)
& =
\sup \left\{ \int \gpot d \mu - \pressure( \mapa, \rpot + \gpot) : \gpot \in C(\spc) \right\}
\\ & =
\sup \left\{ \int \gpot d\mu - \pressure(\mapa, \rpot) - Q_{\rpot}(\gpot) : \gpot \in C(\spc) \right\}
\\ & =
-\pressure(\mapa, \rpot) + \sup \left\{ \int \gpot d\mu - Q_{\rpot}(\gpot) : \gpot \in C(\spc) \right\}
\\ & =
-\pressure(\mapa, \rpot) + {Q_{\rpot}}^*(\mu).
\end{split}
\end{equation*}
This proves part 1. Then part 2  follows from the equalities
\begin{multline*}
 P(\map,\varphi+\psi)-h_\mu(\map)-\int(\varphi+\psi)d\mu
\\ =
Q_\varphi(\psi) + P(\map,\varphi) -h_\mu(\map) -\int(\varphi+\psi)d\mu
\\ =
Q_\varphi(\psi) + {Q_{\rpot}}^* (\mu)-\int\psi d\mu.
\end{multline*}
\end{proof}

\begin{proof}[Proof of Theorem~\ref{t:LDP}]
Let~$\cX$ be the space of all linear functionals on~$\cW$ endowed with the $\cW$-topology, \textit{i.e.} the coarsest topology such that for each $\gpot \in \cW$ the evaluation map $\hgpot : {\cX} \to \R$ defined by $\hgpot(u) = u(\gpot)$ is continuous. Note that $\cX$ is a locally convex real topological vector space with topological dual~$\cX^*=\{ \hgpot \mid \gpot \in \cW \}$.
Given~$\mu \in \mspc$ denote by~$\pi(\mu)$ the element of~$\cX$ such that for each~$\gpot \in \cW$ we have~$\pi(\mu)(\gpot) = \int \gpot d\mu$, and let~${\sM}_{\cW}(X)$ denote the image of the function~$\pi : \mspc \to \cX$ so defined.
Since by hypothesis~$\cW$ is a dense subspace of~$C(\spc)$, the map~$\pi$ is an homeomorphism from ${\sM}(X)$ onto~${\sM}_{\cW}(X)$;
in particular, ${\sM}_{\cW}(X)$ is a compact subset of~$\cX$.
We shall prove the large deviation principle for the sequence~$(\pi[\Omega_n])_{n \ge 1}$  in~${\sM}_{\cW}(X)$, and the corresponding statement for~$(\Omega_n)_{n \ge 1}$ in~${\sM}(X)$ will follow from the fact that~$\pi$ is a homeomorphism.
Let~$\overline{L}$ be  the restriction to~$\cX^*$ of the large deviation functional  associated to~$(\pi[\Omega_n])_{n \ge 1}$,  seen as a sequence of measures on $\cX$;
 recall that for~$\psi \in \cW$ for which the~$\limsup$ defining~$\overline{L}(\widehat{\psi})$ is a limit, we denote~$\overline{L}(\widehat{\psi})$ by~$L(\widehat{\psi})$   (see \S~\ref{prelim-LDP}).
By~\eqref{e:functional equality} we have for each $\psi\in \cW$,
\begin{equation}\label{e:functional equality 2}
{L}(\widehat{\psi}) = \lim_{n \to + \infty}  \frac{1}{n}\log\int_{\cX}\exp \left( n \widehat{\gpot}\right) d \pi[\ldpm_n]= Q_{\rpot}(\gpot).
\end{equation}
Since  $\cW$ is dense in~$C(X)$ and~$Q_{\rpot}$ is continuous (Lemma~\ref{l:LDP}), we get
for each $\mu\in\sM(X)$,
\begin{equation}\label{e:functional equality 4}
L^*(\pi(\mu))= Q_{\rpot}^*(\mu).
\end{equation}
The hypotheses on $\cW$ imply that $Q_{\rpot}$ is Gateaux differentiable at each~$\psi \in \cW$ by part 2 of Lemma~\ref{l:LDP} \cite[Proposition 5.3]{EkeTem99}, which by~\eqref{e:functional equality 2} is equivalent to the Gateaux differentiability of~$L$ on~$\cX^*$.
It follows that all the hypotheses of~\cite[Corollary~4.6.14]{DemZei98} applied to  the sequence ~$(\pi[\Omega_n])_{n \ge 1}$  are verified, and consequently~$(\pi[\Omega_n])_{n \ge 1}$ satisfies a large deviation principle in~$\cX$ with rate function~$L^*$.
Since $\sM_{\cW}(X)$ is closed  in $\cX$ the large deviation principle holds in $\sM_{\cW}(X)$ with rate function
${L^*}_{\mid_{\sM_{\cW}(X)}}$
\cite[Lemma~4.1.5]{DemZei98}, and thus with rate  function $Q_{\rpot}^*\circ\pi^{-1}=\rrate\circ\pi^{-1}$ by~\eqref{e:functional equality 4}.

To prove that~$(\ldpm_n)_{n \ge 1}$ converges to the Dirac mass at the unique equilibrium state~$\mu_{\rpot}$ of~$\mapa$ for the potential~$\rpot$, let~$\sG$ be an open neighborhood of~$\mu_{\rpot}$ in~$\mspc$.
Since~$\rrate$ is lower semi-continuous, non-negative, and it vanishes precisely on~$\{ \mu_{\rpot} \}$ (Lemma~\ref{l:LDP}), the infimum of~$\rrate$ on $\sF \= \mspc \setminus \sG$ is attained at some point of~$\sF$, and thus $\inf_{\sF} \rrate > 0$.
Therefore we have
$$ \limsup_{n \to + \infty} \frac{1}{n} \log \ldpm_n(\sF)
\le
- \inf_{\sF} \rrate
< 0,
$$
and $\lim_{n \to + \infty} \ldpm_n(\sG) = 1$.

To prove the last statement of the theorem, let $\sG \subset \mspc$ be a convex and open set containing an invariant measure~$\mu'$.
Since the function~$\rrate$ is lower semi-continuous, and since it takes finite values precisely on the compact set~$\imspc$ (Lemma~\ref{l:LDP}), there exists  $\mu \in \overline{\sG}\cap\imspc$ such that $\rrate(\mu)=\inf_{\overline{\sG}}\rrate$.
For each $t \in (0, 1)$ put $\mu_t = (1 - t) \mu + t \mu'$, and note that $\mu_t \in \imspc$ and $\mu_t \in \sG$ \cite[1.1, p.~38]{Sch71}.
Since the function~$\rrate$ is affine on~$\imspc$, we have
$$
\inf_{\sG} \rrate
\le
\lim_{t \to 0} \rrate(\mu_t)
=
\rrate(\mu)
=
\inf_{\ov{\sG}} \rrate.
$$
This shows that $\inf_{\sG} \rrate = \inf_{\overline{\sG}}\rrate$.
That is,~$\sG$ is a $I^{\varphi}$\nobreakdash-continuity set and the last assertion of the theorem follows (see~\S\ref{s:preliminaries}).
\end{proof}

\begin{rema}\label{remark-equivalence-of-condition-in-theo-C}
 The equality \eqref{e:functional equality}  is in fact necessary in order to have the  large deviation principle  with rate function $I^{\varphi}$. Indeed, when such a large deviation principle holds,
  Varadhan's theorem (\cite[Theorem 4.3.1]{DemZei98}, \cite[Corollary 3.4]{Com03}) states  that for each $\psi\in C(X)$ the limit
  \[\lim_{n \to + \infty} \frac{1}{n} \log \int_{\mspc} \exp \left( n \int \gpot d \mu \right) d \ldpm_n(\mu)\] exists  and satisfies
\begin{multline*}
\lim_{n \to + \infty} \frac{1}{n} \log \int_{\mspc} \exp \left( n \int \gpot d \mu \right) d \ldpm_n(\mu)
\\ =
\sup\left\{\int\psi d\mu-I^{\varphi}(\mu):\mu\in\sM(X)\right\}
\\ =
\sup\left\{\int\psi d\mu-{Q_{\varphi}}^*(\mu):\mu\in\widetilde{\sM}(X)\right\}
=
Q_{\varphi}(\psi).
\end{multline*}
  The upper semi-continuity  of the measure-theoretic entropy  is also necessary  since by definition the rate function is lower semi-continuous.
  \end{rema}

\begin{rema}\label{remark-proof-LDP-Thm}
The part of the  proof of Theorem~\ref{t:LDP} concerning the large deviation principle is a generalization of~\cite[Theorem 2.1]{Kif90};
 indeed, let us  consider a sequence of measures $(\ldpm_n)_{n \ge 1}$ on $\mspc$ whose associated limiting large deviation functional $L(\widehat{\cdot})$ exists on~$\smspc^*$, and for  $\gpot \in C(X)$ put $Q(\gpot)=L(\widehat{\gpot})$.
If we assume that~$Q$ is Gateaux differentiable at each point of a dense vector subspace $\cW$ of $C(X)$ (which is the hypothesis of \cite[Theorem 2.1]{Kif90}, and the one
 of  Theorem~\ref{t:LDP}  with $Q=Q_{\rpot}$), then the proof works verbatim (the convexity of~$L$ gives the convexity of~$Q$, and the uniform continuity of~$L$ gives the continuity of~$Q$ (since $\sup_{\mspc} \widehat{\psi}=\sup_X\psi$); the rate function is~$Q^*$).
We have used \cite[Corollary~4.6.14]{DemZei98} instead of~\cite[Theorem 2.1]{Kif90}, first because it is not clear how the proof of~\cite[Theorem 2.1]{Kif90}, which deals with special nets of measures, extends to general sequences, and second because it emphasizes the role of large deviation theory.
In fact,  \cite[Corollary~4.6.14]{DemZei98} can be thought of as an extension to general locally convex spaces of~\cite[Theorem 2.1]{Kif90}.
Indeed, the latter result deals with~$\smspc$, which can be identified (via the map~$\pi$, and thanks to the fact that~$\cW$ is dense in~$C(X)$) with a subspace of the locally convex space~$\cX$.
The hypotheses of \cite[Theorem 2.1]{Kif90} amount to both the equility~\eqref{e:functional equality 2} and the Gateaux differentiability of~$L$ on~$\cX^*$, which are precisely the hypotheses of \cite[Corollary~4.6.14]{DemZei98}.
\end{rema}

\section{Large deviation principles for TCE rational maps}
\label{s:LDP TCE}

The purpose of this section is to prove Theorem~\ref{t:equilibria}, as well as Theorem~\ref{t:LDP TCE} and its corollaries.
The proof of Theorem~\ref{t:equilibria} is deduced from~\cite[Theorem~8]{Dob0804} in~\S\ref{ss:transfer operator}, after recalling some well known definitions and results about transfer operators.
After giving several equivalent characterizations of the pressure function in~\S\ref{ss:characterizations pressure}, we give the proof of Theorem~\ref{t:LDP TCE} and its corollaries in~\S\ref{ss:proof of LDP TCE}.

\subsection{The transfer operator and conformal measures.}\label{ss:transfer operator}
Fix a rational map $\map : \CC \to \CC$ of degree at least two.
For $y \in \CC$ we denote by $\deg_{\map}(y)$ the local degree of~$\map$ at~$y$.
Given a continuous function $\rpot : J(\map) \to \R$ we denote by~$\sL_\rpot$ the (Ruelle-Perron-Frobenius) \textit{transfer operator}, acting on the space of functions defined on~$J(\map)$ and taking values in~$\R$, by
$$
\sL_\rpot (\gpot) (x)
=
\sum_{y \in \map^{-1}(x)} \deg_\map(y) \exp(\rpot(y)) \gpot(y).
$$
Note that~$\sL_{\rpot}$ acts continuously on the space of continuous functions.
We denote by $\sL_\rpot^*$ the continuous operator acting on~$\smjulia$ by
$$ \int \gpot d \sL_\rpot^*(\eta) = \int \sL_\rpot (\gpot) d \eta. $$
Note that it maps non-zero measures to non-zero measures.
By the change of variable formula it follows that for every Borel measure~$\eta$ and every measurable function $\gpot : J(\map) \to \R$ satisfying $\int |\gpot| d\eta < + \infty$, we have $\int |\sL_\rpot(\gpot)| d\eta < + \infty$ and $\int \gpot d\sL_\rpot^*(\eta) = \int \sL_\rpot (\gpot) d\eta$.

Given a continuous function $g: J(\map) \to [0, + \infty)$ we say that a Borel measure~$\eta$ supported on~$J(\map)$ is $g$\nobreakdash-\textit{conformal for~$\map$} if for every subset~$E$ of~$J(\map)$ on which~$\map$ is injective we have $\eta(\map(E)) = \int_E g d\eta$.

The following lemma is well-known.
Part~2 is a special case of \cite[Proposition~2.2]{DenUrb91a}.
\begin{lemm}\label{l:conformal measures as engienvectors}
Let~$\map$ be a complex rational map, and let $\rpot : J(\map) \to \R$ be a continuous function.
Then the following conclusions hold.
\begin{enumerate}
\item[1.]
There is $\lambda > 0$ and a Borel probability measure~$\eta$ such that $\sL_\rpot^*(\eta) = \lambda \eta$.
\item[2.]
For a given $\lambda > 0$, a Borel measure~$\eta$ supported on~$J(\map)$ is $\lambda \exp(-\rpot)$\nobreakdash-conformal for~$\map$ if and only if $\sL_\rpot^* \eta = \lambda \eta$.
\end{enumerate}
\end{lemm}

\begin{proof}
Let $\hsL_\rpot^*$ be the map acting on~$\mjulia$  defined by
$$
\hsL_\rpot^*(\eta) = (\sL_\rpot^*(\eta)(J(\map)))^{-1} \sL_\rpot^*(\eta).
$$
As~$\hsL_\rpot^*$ is continuous and~$\mjulia$ is compact and convex, the Schauder-Tychonoff theorem \cite[\S{}V.10, Theorem~5]{DunSch88a} then implies that $\hsL_\rpot^*$ has a fixed point~$\eta$.
Letting $\lambda = \sL_\rpot^*(\eta)(J(\map)) > 0$, we have $\sL_\rpot^*(\eta) = \lambda \eta$.

Note that for every Borel probability measure~$\eta$ and every Borel subset~$E$ of~$J(\map)$ on which~$\map$ is injective, we have $\sL_\rpot (\one_E \exp(- \rpot)) = \one_{\map(E)}$ and,
\begin{multline}\label{e:transfer as jacobian}
\eta(\map(E))
=
\int \one_{\map(E)} d\eta
\\ =
\int \sL_\rpot (\one_E \exp(- \rpot)) d\eta
=
\int \one_E \exp(-\rpot) d \sL_\rpot^*(\eta).
\end{multline}
So, if for some $\lambda > 0$ the measure~$\eta$ satisfies $\sL_\rpot^*(\eta) = \lambda \eta$, then~$\eta$ is $\lambda \exp(-\rpot)$\nobreakdash-conformal.
Suppose on the other hand that~$\eta$ is $\lambda \exp(-\rpot)$\nobreakdash-conformal.
Then~\eqref{e:transfer as jacobian} implies that for every Borel subset~$E$ of~$J(\map)$ on which~$\map$ is injective we have
$$
\int \one_E \exp(-\rpot) d \sL_\rpot^* \eta = \lambda \int \one_E \exp(- \rpot) d\eta.
$$
As~$J(\map)$ can be partitioned into a finite number of Borel sets on which~$\map$ is injective, this equality holds in fact for every Borel subset~$E$ of~$J(\map)$.
We thus have $\sL_\rpot^* \eta = \lambda \eta$.
\end{proof}

\begin{proof}[Proof of Theorem~\ref{t:equilibria}]
Let~$\map$ be a rational map satisfying the \TCEC{} and let~$\rpot : J(\map) \to \R$ be a H\"older continuous function.
Since the measure-theoretic entropy of~$\map$ is upper semi\nobreakdash-continuous~\cite{FreLopMan83,Lju83}, it follows that there is an equilibrium state~$\rho$ of~$\map$ for the potential~$\rpot$.
To prove the uniqueness, first observe that by Lemma~\ref{l:conformal measures as engienvectors} there is a $\left(\exp(\pressure(\map, \rpot) - \rpot)\right)$\nobreakdash-conformal probability measure for~$\map$.
On the other hand, by~\cite[Main Theorem]{PrzRivSmi03} the Lyapunov exponent of every invariant measure supported on~$J(\map)$ is positive, so \cite[Theorem~8]{Dob0804} implies that~$\rho$ is the unique equilibrium state of~$\map$ for the potential~$\rpot$.
\end{proof}

\subsection{Characterizations of the pressure function}
\label{ss:characterizations pressure}
Given a rational map~$\map$ satisfying the \TCEC{} and a H\"older continuous function $\rpot : J(\map) \to \R$, in this section we characterize the pressure function~$P(\map, \rpot)$ in terms of iterated preimages (Lemma~\ref{l:tree pressure}), periodic points (Lemma~\ref{l:periodic points}), and Birkhoff averages (Lemma~\ref{l:Birkhoff averages}); compare with~\cite{PrzRivSmi04}.
Lemma~\ref{l:tree pressure} is also used in Appendix~\ref{s:Ruelle-Perron-Frobenius}.
Compare Lemma~\ref{l:Birkhoff averages} with~\cite[Theorem~3]{Lop90}.

We will make use of the following equivalent formulation of the \TCEC~\cite[Main Theorem]{PrzRivSmi03}, for a rational map~$\map$ of degree at least two.
\begin{generic}[Exponential Shrinking of Components (\ExpShrink)] There exist $\la_{\ES}>1$ and $r_0 > 0$ such that for every $x \in J(\map)$, every integer $n \ge 1$ and every connected component~$W$ of $\map^{-n}(B(x, r_0))$ we have
$$
\diam(W) \le \la_{\ES}^{-n}.
$$
\end{generic}

Recall that for each integer~$n \ge 1$, and each $\gpot : J(\map) \to \R$ we denote
$$ S_n(\gpot) = \gpot + \gpot \circ \map + \cdots + \gpot \circ \map^{n - 1}. $$

\begin{lemm}\label{l:tree pressure}
Let~$\map$ be a rational map satisfying the \TCEC{} and let $\rpot : J(\map) \to \R$ be a H\"older continuous function.
Then there is a constant~$C_0 > 0$ such that for each $x \in J(\map)$ we have
$$ C_0^{-1}
\le
\exp(- n\pressure(\map, \rpot)) \cdot \sL_\rpot^n(\one)(x)
\le
C_0. $$
In particular,
\[
\lim_{n \to + \infty} \frac{1}{n} \log \sL_\rpot^n(\one)(x)
=
\pressure (\map, \rpot).
\]
\end{lemm}
\begin{proof}
Let $\lambda_{\ES} > 1$ and $r_0 > 0$ be as in the \ExpShrink{} condition.
Let $\kappa \in (0, 1)$ be the exponent of~$\rpot$.
We will use the following fact several times:
If $x, x' \in J(\map)$ belong to a ball~$B$ of radius less than or equal to~$r_0$
centered at $J(\map)$, then for every $y \in \map^{-n}(x)$ and $y' \in
\map^{-n}(x')$ in the same connected component of~$\map^{-n}(B)$, we have
$$
|S_n(\rpot)(y) - S_n(\rpot)(y')|
\le
|\rpot|_{\kappa} (2r_0)^\kappa (\lambda_{\ES}^\kappa - 1)^{-1}.
$$
So, if we put $C_1 \= \exp(|\rpot|_{\kappa}(2r_0)^\kappa (\lambda_{\ES}^\kappa -
1)^{-1})$, then we have
$$
C_1^{-1} \le \sL^n_\rpot(\one)(x) / \sL^n_\rpot(\one)(x') \le
C_1.
$$

Let~$\sU$ be a finite covering of~$J(\map)$ by balls of radius $r_0$
centered at~$J(\map)$.

\partn{1}
We will show that there is a constant~$C_0 > 1$ so that for every integer~$n \ge 1$, and every $x, x' \in J(\map)$ we have
$$
C_0^{-1} \le \sL_\rpot^n(\one)(x) / \sL_\rpot^n(\one)(x') \le
C_0.
$$

By the locally eventually onto property of~$\map$ on~$J(\map)$, there is a positive integer~$n_0$ such that for every $B \in \sU$ we
have $J(\map) \subset \map^{n_0}(B)$.
We will show that for each $n \ge n_0$ and $x \in J(\map)$ we have
\begin{multline*}
C_1^{-1} \left( \sup_{J(\map)} \sL_\rpot^{n - n_0}(\one) \right) \left( \inf_{J(\map)} \exp(\rpot) \right)^{n_0}
\le
\sL_\rpot^n(\one)(x)
\\ \le
\deg(\map)^{n_0} \left( \sup_{J(\map)} \sL_\rpot^{n - n_0}(\one) \right) \left( \sup_{J(\map)} \exp(\rpot) \right)^{n_0}.
\end{multline*}
The desired assertion follows easily from these inequalities.

The second inequality is an easy consequence of the formula,
$$
\sL_\rpot^n(\one)(x)
=
\sum_{y \in \map^{-n_0}(x)} \deg_{\map^{n_0}}(y) \exp(S_{n_0}(\rpot)(y)) \sL_\rpot^{n - n_0}(\one)(y).
$$
and from the fact that $\# \left(\map^{-n_0}(x)\right) \le \deg(\map)^{n_0}$.
To prove the first inequality, let $y_0 \in J(\map)$ be such that
$$ \sL_\rpot^{n - n_0}(\one) (y_0) = \sup_{J(\map)} \sL_\rpot^{n - n_0}(\one). $$
Furthermore, let $B \in \sU$ containing~$y_0$, and let $y \in B$ be such that $\map^{n_0}(y) = x$.
Then we have
$$
\sL_\rpot^n(\one)(x)
\ge
\exp(S_{n_0}(\rpot)(y)) \sL_\rpot^{n - n_0}(\one) (y)
\ge
\left( \inf_{J(\map)} \exp(\rpot) \right)^{n_0} C_1^{-1} \sL_\rpot^{n - n_0}(\one) (y_0).
$$

\partn{2}
We will prove that for each~$x \in J(\map)$ we have
$$ \lim_{n \to + \infty} \frac{1}{n} \log \sL_{\rpot}^n \one (x) = \pressure(\map, \rpot). $$

Given $\delta > 0$ let~$\varepsilon > 0$ and $n_0 \ge 1$ be such that for each~$n \ge n_0$ there is a $(n, \varepsilon)$\nobreakdash-separated set~$\sN$
such that
$$
\sum_{y \in \sN} \exp(S_n(\rpot(y))) \ge \exp(n(\pressure (\map, \rpot) -
\delta)).
$$
Taking~$n_0$ larger if necessary, we assume that $\lambda_{\ES}^{n_0} \le \varepsilon$.

Fix~$n \ge n_0$, let~$\sN$ be as above, and let $B \in \sU$ be such that the set $\sN_B \= \{ y \in \sN \mid \map^{n + n_0}(y) \in B \}$ satisfies
$$
\sum_{y \in \sN_B} \exp(S_n(\rpot)(y))
\ge
\frac{1}{\# \sU} \exp(n( \pressure (\map, \rpot) - \delta)).
$$
Since for each $m = n_0, n_0 + 1, \ldots, n$ the diameter of each connected component of~$\map^{-m}(B)$ is less than or equal to $\lambda_{\ES}^{m} \le \varepsilon$, it follows that each connected component of~$\map^{-(n + n_0)}(B)$ can contain at most one element of~$\sN$.
Therefore for each $x \in B \cap J(\map)$ we have
\begin{multline*}
\sL_\rpot^{n + n_0}(\one)(x)
\ge
C_1^{-1} \left( \inf_{J(\map)} \exp(\rpot) \right)^{n_0} \sum_{y \in \sN_B} \exp(S_n(\rpot)(y))
\\ \ge
C_1^{-1} \left( \inf_{J(\map)} \exp(\rpot) \right)^{n_0} \frac{1}{\# \sU} \exp(n(\pressure(\map, \rpot) - \delta)).
\end{multline*}
Together with part~1 this implies that for each $x' \in J(\map)$ we have
$$
\liminf_{n \to + \infty} \frac{1}{n} \log \sL_\rpot^n(\one)(x')
\ge
\pressure(\map, \rpot) - \delta.
$$
Since $\delta > 0$ was arbitrary, this shows that for each  $x' \in J(\map)$ we have
$$
\liminf_{n \to + \infty} \frac{1}{n} \log \sL_\rpot^n(\one)(x')
\ge
\pressure(\map, \rpot).
$$

It remains to prove that for each $x \in J(\map)$ we have,
\begin{equation}\label{e:lower pressure}
\limsup_{n \to + \infty} \frac{1}{n} \log \sL_\rpot^n(\one)(x) \le
\pressure(\map, \rpot).
\end{equation}
Let $\varepsilon > 0$ be given. For each $n \ge 1$ and $y_0 \in
J(\map)$ denote by $N_n(y_0)$ the number of points in
$\map^{-n}(\map^n(y_0))$, counted with multiplicity, that are $(n,
\varepsilon)$\nobreakdash-close to~$y_0$, and put $N_n \= \sup_{y_0
\in J(\map)} N_n(y_0)$. Then, for every $n \ge 1$ and $x_0 \in J(\map)$
the set $\map^{-n}(x_0)$ can be partitioned into at most~$N_n$ sets,
each of which is $(n, \varepsilon)$\nobreakdash-separated. It
follows that $\map^{-n}(x_0)$ contains a subset~$\sN$ that is $(n,
\varepsilon)$\nobreakdash-separated and such that
$$
\sum_{y \in \sN} \exp(S_n(\rpot(y)))
\ge
\frac{1}{N_n} \sL_\rpot^n(\one)(x_0).
$$
Thus, to prove inequality~\eqref{e:lower pressure} it is enough to
prove that $\limsup_{n \to + \infty} \tfrac{1}{n} \log N_n$ can be
made arbitrarily small by taking~$\varepsilon$ sufficiently small.

Let $L \ge 1$ be given. As none of the critical points of~$\map$ in
$J(\map)$ is periodic, there is $\varepsilon > 0$ such that for every
$c \in \Crit(\map) \cap J(\map)$, $x \in B(c, 2\varepsilon)$ and $j \in
\{1, \ldots, L \}$ we have $\map^j(x) \not\in B(c, 2\varepsilon)$.
Reducing~$\varepsilon$ if necessary we assume that for every $x \in
J(\map)$ such that $\dist(x, \Crit(\map) \cap J(\map)) \ge 2\varepsilon$, the
map $\map$ is injective on~$B(x, \varepsilon)$.

For each $y \in J(\map)$ put $N_0(y) = 1$. Note that if $y_0 \in J(\map)$
and $y \in \map^{-n}(\map^n(y_0))$ are $(n,
\varepsilon)$\nobreakdash-close, then~$\map(y)$ and $\map(y_0)$ are $(n -
1, \varepsilon)$\nobreakdash-close. So we have $N_n(y_0) \le \deg(\map)
N_n(\map(y_0))$, and when~$\map$ is injective on $B(y_0, \varepsilon)$ we
have $N_n(y_0) = N_{n - 1}(\map(y_0))$. In particular we have $N_n(y_0)
= N_{n - 1}(\map(y_0))$ when $\dist(y_0, \Crit(\map) \cap J(\map)) \ge
2\varepsilon$. By induction and the definition of~$L$ we obtain that
$$
N_n(y_0) \le \deg(\map)^{\#(\Crit(\map) \cap J(\map))(1 + n/L)},
$$
and that
$$
\limsup_{n \to + \infty} \tfrac{1}{n} N_n(y_0) \le L^{-1} \#
(\Crit(\map) \cap J(\map)) \log \deg(\map).
$$
As we can take~$L$ arbitrarily large by taking $\varepsilon > 0$
sufficiently close to~$0$, this completes the proof of the desired assertion.

\partn{3}
We will complete the proof of the lemma.
By part~1 of Lemma~\ref{l:conformal measures as engienvectors} there is $\lambda > 1$ and a probability measure~$\eta$ such  that $\sL_\rpot^*(\eta) = \lambda \eta$.
Then for every integer $n \ge 1$ we have
\[
\int \sL_\rpot^n(\one) d\eta = \int \one d (\sL_\rpot^*)^n
(\eta) = \lambda^n.
\]
Thus, by part~1 we have that for every $x \in J(\map)$,
$$
C_0^{-1} \la^n \le \sL_\rpot^n(\one)(x) \le C_0 \la^n.
$$
Part~2 implies then that $\lambda = \exp(\pressure(\map, \rpot))$.
\end{proof}
\begin{lemm}\label{l:periodic points}
Let~$\map$ be a rational map satisfying the \TCEC, and for each integer $n \ge 1$ put $\Per_n = \{ p \in J(\map) \mid \map^n(p) = p \}$.
Then for every H\"older continuous function~$\rpot : J(\map) \to \R$ we have
$$
\lim_{n \to + \infty} \frac{1}{n} \log \sum_{p \in \Per_n} \exp(S_n(\rpot)(p))
=
\pressure(\map, \rpot).
$$
\end{lemm}
\begin{proof}
  Let $\lambda_{\ES} > 1$ and $r_0 > 0$ be as in the \ExpShrink{} condition, and let $\kappa \in (0, 1)$ be the exponent of~$\rpot$.
Just as in the proof of Lemma~\ref{l:tree pressure} we will use the following fact several times:
Letting $C \= |\rpot|_{\kappa} (2r_0)^\kappa (\lambda_{\ES}^\kappa - 1)^{-1}$, for each $x, x' \in J(\map)$ that belong to a ball~$B$ of radius less than or equal to~$r_0$ centered at $J(\map)$, and for each $y \in \map^{-n}(x)$ and $y' \in
\map^{-n}(x')$ in the same connected component of~$\map^{-n}(B)$, we have $|S_n(\rpot)(y) - S_n(\rpot)(y')| \le C$.

Let $n_0 \ge 1$ be sufficiently large so that $\la_{\ES}^{n_0} < r_0/3$, and fix $n \ge n_0$.

Let~$F$ be a finite subset of~$J(\map)$ that is $(r_0/3)$\nobreakdash-dense in $J(\map)$.
Let $x \in F$ and let~$W$ be a connected component of~$\map^{-n}(B(x, r_0))$ intersecting~$B(x, r_0/3)$.
We have $W \subset B(x, 2r_0/3)$, so the number of elements of~$\Per_n$ contained in~$W$ is the same as the number of elements of~$\map^{-n}(x_0)$ in~$W$, counted with multiplicity.
Considering that each element of~$\Per_n$ is contained in such a~$W$, we conclude that
$$
\sum_{p \in \Per_{n}} \exp(S_n(\rpot)(p))
\le
\exp(C) \sum_{x \in F} \sL_{\rpot}^n (\one) (x). $$
Then Lemma~\ref{l:tree pressure} implies that
$$
\limsup_{n \to + \infty} \frac{1}{n} \log \sum_{p \in \Per_n} \exp(S_n(\rpot)(p))
\le
\pressure(\map, \rpot).
$$

Fix $x_0 \in J(\map)$ and let~$m_0 \ge 1$ be sufficiently large so that $J(\map) \subset \map^{m_0}(B(x_0, r_0/3))$.
Let~$W$ be a connected component of~$\map^{-n}(B(x_0, r_0))$.
Then there is a connected component~$W_0$ of~$\map^{-m_0}(W)$ intersecting~$B(x_0, r_0/3)$.
Since~$W_0$ is a connected component of~$\map^{- (n + m_0)}(B(x_0, r_0))$, we have $W_0 \subset B(x_0, 2r_0/3)$.
So, if we denote by~$D_0$ the degree of~$\map^{n + m_0} : W_0 \to B(x_0, r_0)$, then~$W_0$ contains precisely~$D_0$ elements of~$\Per_{n + m_0}$.
Since the degree of~$\map^n : W \to B(x_0, r_0)$ is less than or equal to~$D_0$, letting $C' = \exp \left(C +  m_0 \sup_{\CC} |\rpot| \right)$, we have
\begin{equation*}
\sum_{x \in W \cap \map^{-n}(x_0)} \deg_\map(x) \exp(S_n(\rpot))(x)
\le
C' \sum_{p \in W_0 \cap \Per_{n + m_0}} \exp(S_n(\rpot)(p)).
\end{equation*}
We thus have
$$
\sum_{p \in \Per_{n + m_0}} \exp(S_n(\rpot)(p))
\ge
(C')^{-1}\sL_{\rpot}^n (\one) (x_0), $$
and Lemma~\ref{l:tree pressure} implies that
$$
\liminf_{n \to + \infty} \frac{1}{n} \log \sum_{p \in \Per_n} \exp(S_n(\rpot)(p))
\ge
\pressure(\map, \rpot).
$$
\end{proof}
\begin{lemm}\label{l:Birkhoff averages}
Let~$\map$ be a complex rational map satisfying the \TCEC, let $\rpot : J(\map) \to \R$ be a H\"older continuous function, and let $\mu_\rpot$ be the unique equilibrium state of~$\map$ for the potential~$\rpot$.
Then for every H\"older continuous function $\gpot : J(\map) \to \R$ we have
$$ \lim_{n \to + \infty} \frac{1}{n} \log \int \exp (S_n(\gpot)) d \mu_{\rpot}
=
\pressure(\map, \rpot + \gpot) - \pressure(\map, \rpot). $$
\end{lemm}
\begin{proof}
Let~$\eta_\rpot$ be the $(\rpot - \pressure(\map, \rpot))$\nobreakdash-conformal measure of~$\map$, and let~$h_\rpot$ be the Radon-Nikodym derivative of~$\mu_\rpot$ with respect to~$\eta_\rpot$.
Since $\inf h_\rpot > 0$ and $\sup h_\rpot < + \infty$, it is enough to prove the limit with~$\mu_\rpot$ replaced by~$\eta_\rpot$.

For each integer $n \ge 1$ we have
\begin{multline*}
  \int \exp (S_n(\gpot)) d\eta_{\rpot}
=
\int \exp (S_n(\gpot)) d \left( \exp (-n P(\map, \rpot)) \sL_{\rpot}^{*n}\eta_{\rpot} \right)
\\ =
\exp (- n P(\map, \rpot)) \int \sL_{\rpot}^n \left( \exp (S_n(\gpot)) \right) d \eta_{\rpot}
\end{multline*}
Using $\sL_{\rpot}^n \left( \exp (S_n(\gpot)) \right) = \sL_{\rpot + \gpot}^n \one$, the assertion of the proposition is then a direct consequence of Lemma~\ref{l:tree pressure}.
\end{proof}

\subsection{Proof of Theorem~\ref{t:LDP TCE} and its corollaries}\label{ss:proof of LDP TCE}
\begin{proof}[Proof of Theorem~\ref{t:LDP TCE}]
First recall that the topological entropy of~$\map$ is equal to $\log \deg(\map)$ \cite{Gro03,Lju83}, and that the measure-theoretic entropy of~$\map$ is upper semi-continuous~\cite{FreLopMan83,Lju83}.
Fix a H\"older continuous function $\gpot : J(\map) \to \R$.
For the sequence~$(\ldpm_n)_{n \ge 1}$ associated to periodic points we have,
\begin{multline*}
\int_{\mjulia} \exp \left( n \int \gpot d\mu \right) d \ldpm_n(\mu)
\\ =
\frac{\sum_{p \in \Per_n} \exp(S_n(\rpot)(p)) \exp \left( n \int \gpot d\average_n(p) \right)}{\sum_{p' \in \Per_n} \exp(S_n(\rpot)(p'))}
\\ =
\frac{\sum_{p \in \Per_n} \exp(S_n(\rpot + \gpot)(p))}{\sum_{p' \in \Per_n} \exp(S_n(\rpot)(p'))}.
\end{multline*}
Analogously, for the sequence~$(\ldpm_n(x_0))_{n \ge 1}$ associated to the iterated preimages of a point $x_0 \in J(\map)$, we have
\begin{equation*}
\int_{\mjulia} \exp \left( n \int \gpot d\mu \right) d \ldpm_n(x_0)(\mu)
=
\frac{\sum_{x \in \map^{-n}(x_0)} \exp(S_n(\rpot + \gpot)(x))}{\sum_{y \in \map^{-n}(x_0)} \exp(S_n(\rpot)(y))}.
\end{equation*}
Finally, for the sequence~$(\ldpmb_n)_{n \ge 1}$ associated to the Birkhoff averages we have,
\begin{equation*}
\int_{\mjulia} \exp \left( n \int \gpot d\mu \right) d \ldpmb_n(\mu)
=
\int_{J(\map)} \exp (S_n(\gpot)) d \mu_{\rpot}.
\end{equation*}
Therefore, \eqref{e:functional equality}  holds with~$\gpot$ for the sequences $(\ldpm_n)_{n \ge 1}$, $(\ldpm_n(x_0))_{n \ge 1}$, and $(\ldpmb_n)_{n \ge 1}$, by Lemma~\ref{l:periodic points}, Lemma~\ref{l:tree pressure}, and Lemma~\ref{l:Birkhoff averages}, respectively.
Consequently,  all the assertions  of Theorem~\ref{t:LDP TCE} follow from Theorem~\ref{t:equilibria} and   Theorem~\ref{t:LDP}.
\end{proof}

\begin{proof}[Proof of Corollary~\ref{c:contracted LDP}]
The first assertion  is obtained from Theorem~\ref{t:LDP TCE} applying  the contraction principle  with the map $\widehat{\psi}$ (see \S~\ref{s:preliminaries}).
For each $\delta>0$ put
$$ \sG_{1,\delta}
=
\left\{ \mu \in \mjulia : \int \gpot d\mu > \delta \right\}
$$
and
$$
\sG_{2,\delta}
=
\left\{ \mu \in \mjulia : \int \gpot d\mu < -\delta \right\}.
$$
If there exists some $\delta_0>0$ such that  $(\sG_{1,\delta_0}\cup\sG_{2,\delta_0})\cap\imjulia\neq\emptyset$, then~\eqref{c:contracted LDP-eq2}  follows from the last statement of Theorem~\ref{t:LDP TCE} for all $\varepsilon \in (0, \delta_0]$.
Moreover, the value of~\eqref{c:contracted LDP-eq2} is strictly negative since by hypothesis $\mu_\varphi\not\in\overline{\sG_{1,\varepsilon}\cup\sG_{2,\varepsilon}}$.
Assume now that for all $\delta > 0$ we have
$(\sG_{1,\delta}\cup\sG_{2,\delta})\cap\imjulia=\emptyset$.
Since for each~$\delta > 0$ and~$j \in \{1, 2 \}$ we have $\overline{\sG_{j, 2 \delta}} \subset \sG_{j, \delta}$, we obtain
\[\overline{\sG_{1,\delta}\cup \sG_{2,\delta}}\cap\imjulia=\emptyset, \]
 and the conclusion follows from  the large deviation  upper bounds applied to $\overline{\sG_{1,\delta}\cup \sG_{2,\delta}}$ for all $\delta>0$ (so that both sides of~\eqref{c:contracted LDP-eq2} are $-\infty$).
\end{proof}

\begin{proof}[Proof of Corollary~\ref{c:entropy expression}]
The first (resp. second) equality is a direct consequence of the definition of the rate function~\eqref{t:LDP TCE-eq1} together with~\eqref{t:LDP TCE-eq2}, and Lemma~\ref{l:periodic points} (resp. Lemma~\ref{l:tree pressure}).
The last equality follows from~\eqref{t:LDP TCE-eq1} and~\eqref{t:LDP TCE-eq2}.
\end{proof}

\appendix

\section{A Ruelle-Perron-Frobenius type  theorem for TCE rational maps}\label{s:Ruelle-Perron-Frobenius}
The purpose of this appendix is to give an alternative proof of Theorem~\ref{t:equilibria}, as a direct consequence of the following Ruelle-Perron-Frobenius type theorem; compare with~\cite{DenUrb91b,Prz90,DenPrzUrb96}.
\begin{theoalph}\label{t:Ruelle-Perron-Frobenius}
Let~$\map$ be a rational map satisfying the \TCEC{} and let $\rpot : J(\map) \to \R$ be a H\"older continuous function.
Then the following conclusions hold.
\begin{enumerate}
\item[1.]
There is a unique probability measure~$\eta_0$ that is supported on~$J(\map)$ and that satisfies
$$ \sL_{\rpot}^* \eta_0 = \exp(\pressure(\map, \rpot)) \eta_0. $$
More generally, if for some~$\la > 1$ there is a probability measure~$\eta$ supported on~$J(\map)$ and such that $\sL_{\rpot}^* \eta = \la \eta$, then $\la = \exp(\pressure(\map, \rpot))$ and $\eta = \eta_0$.

In particular~$\eta_0$ is the unique $\left(\exp(\pressure(\map, \rpot) - \rpot)\right)$\nobreakdash-conformal probability measure for~$\map$ supported on~$J(\map)$.
\item[2.]
There is a unique H\"older continuous function~$h_0 : J(\map) \to (0, + \infty)$ satisfying
$$ \sL_\rpot h_0 = \exp(\pressure(\map, \rpot)) h_0 \text{ and } \int h_0 d \eta_0 = 1. $$
Furthermore, the probability measure~$h_0 \eta_0$ is invariant by~$\map$ and it is the unique equilibrium state of~$\map$ for the potential~$\rpot$.
\end{enumerate}
\end{theoalph}

To prove this result we first consider the following lemma, which is precisely~\cite[Part~1 of Lemma~3.3]{PrzRiv07}.
\begin{lemm}
\label{l:expansion}
Let~$\map$ be a rational map satisfying the \ExpShrink{} condition with constants $\la_{\ES} > 1$ and $r_0 > 0$.
Then there are constants~$\theta_0 \in (0, 1)$ and $C_0 > 0$ such that for each $x \in J(\map)$, each $r \in (0, r_0)$, and each connected component~$W$ of~$\map^{-n}(B(x, r))$, we have
$$ \diam(W) \le C_0 \la_{\ES}^{-n} r^{\theta_0}. $$
\end{lemm}

We denote by~$\| \cdot \|_{\infty}$ the supremum norm on the space of real continuous functions defined on~$J(\map)$.
Given $\alpha \in (0, 1]$ we will say that a function~$\rpot : J(\map) \to \R$ is \textit{H\"older continuous with exponent~$\alpha$} if there is a constant $C > 0$ such that for all $x, y \in J(\map)$ we have
$$ |\rpot(x) - \rpot(y)| \le C \dist(x, y)^\alpha. $$
For such a function~$\rpot$ we put
$$ | \rpot |_\alpha = \sup \{ |\rpot(x) - \rpot(y)|\dist(x, y)^{- \alpha}:  x, y \in J(\map) \text{ distinct} \}, $$
and $\| \rpot \|_\alpha = \| \rpot \|_\infty + | \rpot |_\alpha$.
Note that $\| \rpot \|_\alpha$ defines a norm on the space of H\"older continuous functions with  exponent~$\alpha$.
\begin{lemm}
\label{l:three norm inequality}
Let~$\map$ be a rational map satisfying the \ExpShrink{} condition, and let~$\theta_0 \in (0, 1)$ be given by Lemma~\ref{l:expansion}.
Then for every $\alpha \in (0, 1)$, and every H\"older continuous function $\rpot : J(\map) \to \R$ with exponent~$\alpha$, there is a constant $C > 0$ such that for every~$\beta \in (0, \alpha]$, every H\"older continuous function~$\gpot : J(\map) \to \R$ with exponent~$\beta$, and every integer~$n \ge 1$ we have
$$ \| \sL_{\rpot}^n \gpot \|_{\beta \theta_0}
\le
C \exp(n \pressure(\map, \rpot)) \left( \| \gpot \|_\infty + \la_{\ES}^{- n \beta} | \gpot |_{\beta} \right). $$
\end{lemm}
\begin{proof}
Let $\la_{\ES}$ and $r_0 > 0$ be as in the \ExpShrink{} condition.

Let~$x, x' \in J(\map)$ be outside of the forward orbits of the critical points of~$\map$, and fix an integer~$n$.
Observe that each connected component of~$\map^{-n}(B(x, r_0))$ contains the same number of elements of~$\map^{-n}(x)$ and of~$\map^{-n}(x')$.
Therefore there is a bijection $\iota : \map^{-n}(x) \to \map^{-n}(x')$ such that for every $y \in \map^{-n}(x)$, both~$y$ and $\iota(y)$ belong to the same connected component of~$\map^{-n}(B(x, r_0))$.
In particular we have
$$ \dist(y, \iota(y)) \le C_0 \lambda_{\ES}^{-n} \dist(x, x')^{\theta_0}. $$
Using Lemma~\ref{l:expansion}, we obtain that for each~$y \in \map^{-n}(x)$ we have
$$ | S_n(\rpot)(y) - S_n(\rpot)(\iota(y)) |
\le
|\rpot|_{\alpha} C_0^{\alpha} (\la_{\ES}^{\alpha} - 1)^{-1} \dist(x, x')^{\theta_0\alpha}. $$
So, if we put $C_1 = |\rpot|_{\alpha} C_0^{\alpha} (\la_{\ES}^{\alpha} - 1)^{-1}$, then we have
$$ | \exp(S_n(\rpot)(y)) - \exp(S_n(\rpot)(y')) |
\le
\exp(C_1 r_0^{\theta_0 \alpha}) C_1 \exp(S_n(\rpot)(y)) \dist(x, x')^{\theta_0\alpha}. $$

Using this inequality we obtain,
\begin{multline*}
|\sL_{\rpot}^n \gpot(x) - \sL_{\rpot}^n \gpot(x')|
\le \\ \le
\left| \sum_{y \in \map^{-n}(x)} (\exp(S_n(\rpot)(y)) - \exp(S_n(\rpot)(\iota(y)))) \gpot(y) \right|
+ \\ +
\left| \sum_{y \in \map^{-n}(x)} \exp(S_n(\rpot)(y)) (\gpot(y) - \gpot(\iota(y))) \right|
\le \\ \le
\sL_{\rpot}^n |\gpot| (x) \exp(C_1 r_0^{\theta_0 \alpha}) C_1 \dist(x, x')^{\theta_0\alpha}
+ \\ +
\sL_{\rpot} \one(x) |\gpot|_{\beta} C_0^{\beta} \lambda_{\ES}^{-n \beta} \dist(x, x')^{\theta_0 \beta}.
\end{multline*}

Since the union of the forward orbits of critical points of~$\map$ is nowhere dense in~$J(\map)$, we conclude that the last inequality holds for every $x, x' \in J(\map)$.
Then the assertion of the lemma is obtained using Lemma~\ref{l:tree pressure}.
\end{proof}

\begin{proof}[Proof of Theorem~\ref{t:Ruelle-Perron-Frobenius}]
Let~$\alpha \in (0, 1)$ be the exponent of~$\rpot$, and let~$\theta_0 \in (0, 1)$ be given by Lemma~\ref{l:expansion}.

\partn{1}
Let~$\gpot : J(\map) \to \R$ be a given H\"older continuous function with exponent~$\alpha$.
For each integer~$n \ge 1$ put
$$ \gpot_n \= \frac{1}{n} \sum_{k = 0}^{n - 1} \exp(-k \pressure(\map, \rpot)) \sL_{\rpot}^k \gpot. $$
Then Lemma~\ref{l:three norm inequality} implies that the sequence $\left( \| \gpot_n \|_{\alpha \theta_0} \right)_{n \ge 1}$ is bounded from above independently of~$n$.
It follows that there is a sequence of positive integers $( n_j )_{j \ge 1}$, such that $( \gpot_{n_j} )_{j \ge 1}$ converges uniformly to a H\"older continuous function~$\gpot_0$ of exponent~$\alpha \theta_0$.
We thus have
$$ \sL_{\rpot} \gpot_0 = \exp(\pressure(\map, \rpot)) \gpot_0. $$

\partn{2}
Denote by~$h_0$ a function~$\gpot_0$ as in part~1 when $\gpot = \one$.
Lemma~\ref{l:tree pressure} implies that~$h_0$ takes values in $[C_0^{-1}, C_0] \subset (0, + \infty)$.

We will show that for every~$\gpot$ and~$\gpot_0$ as in part~1 the function~$\gpot_0 / h_0$ is constant.
Put
$$ C \= \sup \{ \gpot_0(x) h_0(x)^{-1} : x \in J(\map) \}, $$
and let~$X$ be the compact set of those~$x \in J(\map)$ such that $\gpot_0(x) = C h_0(x)$.
Then for $x \in X$ we have
\begin{multline*}
C \exp(\pressure(\map, \rpot)) h_0(x)
=
\exp(\pressure(\map, \rpot)) \gpot_0(x)
= \\ =
\sum_{y \in \map^{-1}(x)} \deg_{\map}(y) \exp(\rpot(y)) \gpot_0(y)
\le
C \sum_{y \in \map^{-1}(x)} \deg_{\map}(y) \exp(\rpot(y)) h_0(y)
= \\ =
C \exp(\pressure(\map, \rpot)) h_0(x),
\end{multline*}
which implies that~$\map^{-1}(X) \subset X$.
Therefore~$\map^{-1}(X) \subset X$, and by the locally eventually onto property of~$\map$ on~$J(\map)$ we have that $X = J(\map)$.
That is, we have $\gpot_0 = C h_0$, as wanted.

\partn{3}
Let~$\la > 0$ and let~$\eta_0$ be a probability measure supported on~$J(\map)$ such that $\sL_{\rpot}^* \eta_0 = \la \eta_0$.
Part~1 of Lemma~\ref{l:conformal measures as engienvectors} guaranties that there is at least one such~$\la$ and~$\eta_0$.
Note that for every integer~$n \ge 1$ we have
$$ \int \sL_\rpot^n \one d \eta_0 = \int \one d \sL_\rpot^* \eta_0 = \la^n, $$
so Lemma~\ref{l:tree pressure} implies that $\la = \exp(\pressure(\map, \rpot))$ and hence that~$\sL_{\rpot}^* \eta_0 = \exp(\pressure(\map, \rpot)) \eta_0$.

Note that for each~$\gpot$ and~$\gpot_0$ as in part~1 we have $\int \gpot_0 d \eta_0 = \int \gpot d \eta_0$.
In particular, letting~$\gpot = \one$, we obtain that~$\int h_0 d\eta_0 = 1$.
If we denote by~$C > 0$ the constant given by part~2, so that~$\gpot_0 = C h_0$, then we have
$$ \int \gpot d\eta_0 = \int \gpot_0 d\eta_0 = \int Ch_0 d\eta_0 = C. $$
That is, we have shown that for each accumulation point~$\gpot_0$ of the sequence of functions~$( \gpot_n )_{n \ge 1}$ defined in part~1, we have~$\gpot_0 = (\int \gpot d\eta_0) h_0$.
As this property determines~$\eta_0$ uniquely, we conclude that~$\eta_0$ is the unique probability measure~$\eta$ that is supported on~$J(\map)$ and for which there is~$\la > 0$ such that $\sL_\rpot^* \eta = \la \eta$.

\partn{4}
To show that the measure~$h_0 \eta_0$ is invariant by~$\map$, observe that for each continuous function~$\gpot : J(\map) \to \R$ we have
\begin{multline*}
\int \gpot d \map[h_0 \eta_0]
=
\int \gpot \circ \map d h_0 \eta_0
\\ =
\int \gpot \circ \map \cdot h_0 d \left(\exp(- \pressure(\map, \rpot)) \sL_{\rpot}^* \eta_0 \right)
\\ =
\exp(- \pressure(\map, \rpot)) \int \sL_{\rpot} \left(\gpot \circ \map \cdot h_0 \right) d \eta_0
\\ =
\exp(- \pressure(\map, \rpot)) \int \gpot \sL_{\rpot} h_0 d \eta_0
=
\int \gpot h_0 d \eta_0.
\end{multline*}

Let~$\rho$ be an ergodic and invariant probability measure supported on~$J(\map)$.
We will show that~$\rho$ is an equilibrium state of~$\map$ for the potential~$\rpot$ if and only if the measure~$\eta \= h_0^{-1}\rho$ is $(\exp(\pressure(\map, \rpot) - \rpot))$-conformal for~$\map$.
Together with the uniqueness of~$\eta_0$, this implies that the measure~$h_0 \eta_0$ is the unique equilibrium state of~$\map$ for the potential~$\rpot$.

As~$\map$ satisfies the \TCEC, the Lyapunov exponent of~$\rho$ is positive~\cite[Main Theorem]{PrzRivSmi03}, so~$\rho$ admits a generating partition of finite entropy, see for example~\cite{Man83} (where it was assumed that the entropy is positive, but in fact it was only used that the Lyapunov exponent is positive), ~\cite[\S2]{DenUrb91b}, \cite{Dob0804} or~\cite{PUbook}.
This implies that Rokhlin formula holds~\cite[10\S]{Par69}:
$$ h_\rho = \int \log \Jac_\rho d \rho. $$
Using $\Jac_\eta = \frac{h_0 \circ \map}{h_0} \Jac_\rho$, we obtain $ h_\rho = \int \log \Jac_\eta d \rho$, and
\begin{align*}
h_\rho & - \pressure(\map, \rpot) + \int \rpot d \rho
\\ & =
\int \log \left(\Jac_\eta \exp(\rpot - \pressure(\map, \rpot)) \right) d\rho
\\ & \le
\int \Jac_\eta \exp(\rpot - \pressure(\map, \rpot)) d \rho - 1
\\ & =
\exp(- \pressure(\map, \rpot)) \int \sum_{y \in \map^{-1}(x)} \Jac_\rho(y)^{-1} \Jac_\eta(y) \exp(\rpot(y)) d \rho(x) - 1
\\ & =
\exp(- \pressure(\map, \rpot)) \int h_0(x)^{-1} \sum_{y \in \map^{-1}(x)} h_0(y) \exp(\rpot(y)) d \rho(x) - 1
\\ & =
\exp(- \pressure(\map, \rpot)) \int h_0^{-1} \sL_{\rpot} h_0 d \rho - 1
\\ & =
0
\end{align*}
This shows that~$\rho$ is an equilibrium state of~$\map$ for the potential~$\rpot$ if and only if $\Jac_\eta = \exp(\pressure(\map, \rpot) - \varphi)$ holds on a set of full measure with respect to~$\rho$.
Since~$h_0$ takes values on~$[C_0^{-1}, C_0]$, this last condition is equivalent to the condition that~$\Jac_\eta = \exp(\pressure(\map, \rpot))$ holds on a set of full measure with respect to~$\eta$; or equivalently, that~$\eta$ is a $\exp(\pressure(\map, \rpot) - \rpot)$-conformal measure for~$\map$.
\end{proof}

\section{On Kifer's result for semi-flows}\label{s:flows}

In this section we clarify the relation between Theorem~\ref{t:LDP} and  the main result of  \cite{Kif90}  concerning dynamical systems (namely, Theorem 3.4 of that paper). We claim no originality concerning the proofs of Theorem~\ref{t:LDP} and Theorem~\ref{Kifer-theo}, since in both cases the basic ideas are in~\cite{Kif90}.
See~\cite{Ara07} and references therein for large deviation upper-bounds, for some non-uniformly hyperbolic semi\nobreakdash-flows.

Recall that~\cite{Kif90} concerns large deviations  in $\sM(Y)$, where~$Y$ is a  compact metric space that is not necessarily invariant.
We show  how  the large deviation lower bounds in~$\sM(Y)$ can be  recovered, and in fact slightly strengthened (see Remark~\ref{remark-strengthened-Kifer's-theo}), from Theorem~\ref{t:LDP} and Remark~\ref{remark-extension-general-systems}.
In order to get the upper bounds in $\sM(Y)$
we   use  the extension of the variational principe proved in  \cite{Kif90}.
However, if we consider the closure $X$ of the union of the supports of all the invariant probability measures on~$Y$, then $X$ is invariant and the  large deviation principle will be obtained in $\sM(X)$ from  Theorem~\ref{t:LDP}.

The basic ingredients are the following.
Let~$M$ be a locally compact metric space, let~$Y$ be a compact subset of $M$, and let $\mathfrak{T}\in\{
\mathbb{Z}_+,\mathbb{R}_+\}$.
  For each $t\in \mathfrak{T}$  let  $F^t:M\rightarrow M$  be a continuous map, put $Y_t=\{x\in M:F^s(x)\in Y, 0\le s\le t\}$,  $\sM^F_Y=\{\mu\in\sM(Y):F^t[\mu]=\mu, t\ge 0\}$ (i.e. $\sM^F_Y$ is the set of $F^t$-invariant probability measures for all $t\in \mathfrak{T}$), and $X=\overline{\bigcup_{\mu\in\sM^F_Y}\ \textnormal{supp\ }\mu}$.
We shall use the notations of Remark~\ref{remark-extension-general-systems} for the system induced on $X$; more precisely,
let $\tau$ be the action of $\mathfrak{T}$
on $X$ given by $\tau^t=F^t$ for all $t\in \mathfrak{T}$, so that $X$ is
$\tau$-invariant  with
$\sM^{\tau}(X)=\sM^F_Y$.
When  $\sM^F_Y\neq\emptyset$, for each  $\mu\in\sM^F_Y$ let $h^1_{\mu}$ denote the entropy of $F^1$ with respect to $\mu$, and note that
   $h^\tau_{\mu}=h^1_{\mu}$.
For each $\phi\in C(Y)$ let $\tilde{I}^{\phi}$ be the function defined on  $\sM(Y)$ by
$$
\tilde{I}^{\phi}(\mu)
=
\begin{cases}
P^\tau(\phi_{\mid X})-\int \phi d\mu - h^1_{\mu} & \text{if } \mu\in\sM^F_Y;
\\
+\infty & \text{if } \mu \in \sM(Y) \setminus \sM^F_Y.
\end{cases}
$$
Since $\sM^{\tau}(X)=\sM^F_Y$ and $h^\tau_{\mu}=h^1_{\mu}$, by identifying  $\sM(X)$ as a (closed) subset of $\sM(Y)$ we see that
   $\tilde{I}^{\phi}$ coincides on  $\sM(X)$ (and takes infinite value outside) with the function  $I^{\phi_{\mid X}}$  associated to the system $(X,\tau)$ as in Remark~\ref{remark-extension-general-systems}, defined by
   $$
I^{\phi_{\mid X}}(\mu)
=
\begin{cases}
P^\tau(\phi_{\mid X})-\int \phi d\mu - h^\tau_{\mu} & \text{if } \mu\in\sM^\tau(X);
\\
+\infty & \text{if } \mu \in \sM(X) \setminus \sM^\tau(X).
\end{cases}
$$
For each $t\in \mathfrak{T}$  let $\average_t:Y_t\mapsto\sM(Y_t)$ defined by
   $\average_t(x)=\frac{1}{t}\int_0^t \delta_{F^t(x)}ds$ when $\mathfrak{T}=\mathbb{R}_+$,  and $\average_t(x)=\frac{1}{t}\sum_{i=0}^{t-1} \delta_{F^t(x)}$ when  $\mathfrak{T}=\mathbb{Z}_+$.

\begin{theoalph}[Following Kifer]\label{Kifer-theo}
Let  $m\in\sM(Y)$,  let $\phi\in C(Y)$,  and assume that the following conditions hold.
\begin{itemize}
\item[(i)] $\sM^F_Y\neq\emptyset$ and the map $h^1$ on $\sM^F_Y$ is finite and upper semi-continuous;
\item[(ii)] For each  $t\in\mathfrak{T}$, each $x\in X$ and each  $\delta>0$ we have
\[{a_{\delta,t}}^{-1}\le m(U_{\delta,x,t})\exp \left(-t\int_Y\phi d\average_t(x) \right)
\le a_{\delta,t},\]
where
\[U_{\delta,x,t}=\{y\in Y_t: d(F^u (x),F^u (y))\le\delta, 0\le u\le t\}\]
and $a_{\delta,t}$  satisfies
\[\lim_{\delta\rightarrow 0}\lim_{t\rightarrow +\infty} a_{\delta,t}^{1/t}>0.\]
\end{itemize}
The following conclusions hold.
\begin{itemize}
\item [1.] For each closed subset $\sF$ of $\sM(X)$ we have
\[\limsup_{t\rightarrow+\infty}\frac{1}{t}\log m\{x\in X:\average_t(x)\in \sF\}\le-\inf_{\sF} \tilde{I}^\phi.\]
If moreover
\[m(U_{\delta,x,t})\exp \left(-t\int_Y\phi d\average_t(x) \right)
\le a_{\delta,t}\]
for all $t\in\mathfrak{T}$, all $x\in Y_t$ and all $\delta>0$, then we can replace $X$ by $Y$ in the above assertion.
\item [2.] If  there is a  dense vector subspace  $\cW\subset C(X)$ such that for each $\psi\in \cW$ there is a unique measure  $\mu\in\sM(Y)$ realizing the supremum in   $\sup_{\mu\in\sM^F_Y}\{\int (\psi+\phi) d\mu+h^1_{\mu}\}$, then
for each open subset  $\sG$ of  $\sM(Y)$ we have
\begin{multline*}
\liminf_{t\rightarrow+\infty}\frac{1}{t}\log m\{x\in Y:\average_t(x)\in \sG\})
\\ \ge
\liminf_{t\rightarrow+\infty}\frac{1}{t}\log m\{x\in X:\average_t(x)\in \sG\cap\sM(X)\})
\\ \ge
-\inf_{\sG\cap\sM(X)} \tilde{I}^\phi=-\inf_{\sG} \tilde{I}^\phi.
\end{multline*}
\end{itemize}
\end{theoalph}

\begin{proof}
Putting for each $\psi\in C(Y)$ and each $\delta>0$,
 \[\gamma_\delta(\psi)=\sup\{|\psi(y)-\psi(z)|:y\in Y, z\in Y, d(y,z)\le \delta\},\]
 we get for
each maximal $(\delta,t)$-separated set $S_{\delta,t}$ in $Y_t$,
\begin{equation}\label{Kifer-theo-eq2}
\frac{1}{t}\log\sum_{x\in S_{\delta,t}\cap X}m(U_{\delta/2,x,t})\exp \left( t\int_X (\psi-\gamma_\delta(\psi))d\average_t(x)\right)
\end{equation}
\[\le\frac{1}{t}\log\int_{X}\exp\left( t\int_X \psi d\average_t(x)\right) dm(x)\]
\[\le\frac{1}{t}\log\sum_{x\in S_{\delta,t}\cap X}m(U_{\delta/2,x,t})\exp \left( t\int_X
   \psi+\gamma_\delta(\psi))d\average_t(x)\right),\]
 and using (ii)   yields
\begin{equation}\label{Kifer-theo-eq4}
\lim_{t\rightarrow+\infty}\frac{1}{t}\log\int_{X}\exp\left(t\int_X\psi d\average_t(x) \right) dm(x)
=P^{\tau}(\phi_{\mid X}+\psi_{\mid X})+\lim_{\delta\rightarrow 0}\lim_{t\rightarrow +\infty}\frac{1}{t}\log a_{\delta,t}
\end{equation}
(note that $S_{\delta,t}\cap X$ is a maximal $(\delta,t)$-separated set in $X$).
 Taking $\psi=0$ in (\ref{Kifer-theo-eq4})  gives
\[
\lim_{t\rightarrow+\infty}\frac{1}{t}\log m(X)=P^{\tau}(\phi_{\mid X})+\lim_{\delta\rightarrow 0}\lim_{t\rightarrow +\infty}\frac{1}{t}\log a_{\delta,t}>-\infty
\]
which implies  $m(X)>0$; in particular, both sides of the above equality vanish hence
\begin{equation}\label{Kifer-theo-eq4.1}
P^{\tau}(\phi_{\mid X})=-\lim_{\delta\rightarrow 0}\lim_{t\rightarrow +\infty}\frac{1}{t}\log a_{\delta,t}.
\end{equation}
 We put $m_X=m/m(X)$,
  and shall consider  the system $(X,\tau)$ and  the net of image
 measures $({{\average_t}}_{\mid X}[m_X])$ on $\sM(X)$.
First note that the hypothesis $(i)$ gives the upper semi-continuity of the map $h^\tau_{\cdot}$.
 From (\ref{Kifer-theo-eq4}) and (\ref{Kifer-theo-eq4.1}) we obtain for each $\psi\in C(Y)$,
\[
\lim_{t\rightarrow+\infty}\frac{1}{t}\log\int_{\sM(X)}\exp \left(t\int_X \psi d\mu\right)d{{\average_t}}_{\mid X}[m_X]=P^{\tau}(\phi_{\mid X}+\psi_{\mid X})-P^{\tau}(\phi_{\mid X}).
\]
Since any element of $C(X)$ is the restriction of some function in $C(Y)$,
it follows that the general hypotheses of \cite[Theorem 5.2]{Com09}
 hold for the net $({{\average_t}}_{\mid X}[m_X])$.
 Therefore,   we get for  each closed subset $\sF$ of $\sM(X)$,
 \[\limsup_{t\rightarrow+\infty}\frac{1}{t}\log{{\average_t}}_{\mid X}[m_X](\sF)=\limsup_{t\rightarrow+\infty}\frac{1}{t}\log m\{x\in X:\average_t(x)\in\sF\}\]
  \[\le-\inf_{\sF} I^{\phi_{\mid X}}=-\inf_{\sF} \tilde{I}^\phi,\]which proves the first assertion of part 1.
 Assume moreover that
 \begin{equation}\label{Kifer-theo-eq5}
 m(U_{\delta,x,t})\exp \left(-t\int_Y\phi d\average_t(x) \right)
\le a_{\delta,t}
\end{equation}
for all  $t\in\mathfrak{T}$, all $x\in Y_t$ and all $\delta>0$.
 For each $t\in \mathfrak{T}$, let $m_t$ be the measure defined on $Y_t$ by putting $m_t=m/m(Y_t)$,
    and
   let $\overline{L}_Y$  be the large deviation functional associated to the net $(\average_t[m_t])$  (seen as acting on  $\widetilde{\sM}(Y)$). Replacing   $S_{\delta,t}\cap X$ (resp. $X$) by $S_{\delta,t}$ (resp. $Y$)  in  the sums appearing in  (\ref{Kifer-theo-eq2}), and  using (\ref{Kifer-theo-eq5})  together with  the fact that
   the topological pressure of any $\psi\in C(Y)$ coincides with $P^{\tau}({\psi}_{\mid X})$ (\cite[Proposition 3.1]{Kif90})
      we get
 \[
 \overline{L}_Y(\widehat{\psi})=\limsup_{t\rightarrow+\infty}
 \frac{1}{t}\log\int_{Y_t}\exp\left(t\int_Y\psi d\average_t(x)\right)dm(x)
\]
\[\le P^{\tau}(\phi_{\mid X}+\psi_{\mid X})-P^{\tau}(\phi_{\mid X})=Q_{\phi_{\mid X}}(\psi_{\mid X}),
\]
where $Q_{\phi_{\mid X}}$ is the map defined as in Lemma~\ref{l:LDP}; moreover, by (\ref{Kifer-theo-eq4}) and (\ref{Kifer-theo-eq4.1})
  the upper limit is a limit and the inequality is an equality, hence for each $\psi\in C(Y)$,
  \[{L}_Y(\widehat{\psi})=Q_{\phi_{\mid X}}(\psi_{\mid X}).\]
By \cite[Lemma 4.5.3]{DemZei98},  $(\average_t[m_t])$ satisfies the large deviation upper  bounds in $\widetilde{\sM}(Y)$ with the function
\begin{equation}\label{Kifer-theo-eq8}
{L_Y}^*(\mu)=\sup\{\mu(\psi)-L_Y(\widehat{\psi}):\psi\in C(Y)\}=\sup\{\mu(\psi)-Q_{\phi_{\mid X}}(\psi_{\mid X}):\psi\in C(Y)\}
\end{equation}
\[\ge\sup\{\mu_{\mid X}(\psi')-Q_{\phi_{\mid X}}(\psi'):\psi'\in C(X)\}=Q_{\phi_{\mid X}}^*(\mu_{\mid X}).\]
Since  $\sM(Y)$ is closed in $\widetilde{\sM}(Y)$, the large deviation principle holds  in   $\sM(Y)$ with rate function  ${{L_Y}^*}_{\mid \sM(Y)}$.
Since the inequality in (\ref{Kifer-theo-eq8}) is an equality  when $\mu\in\sM(X)$,
   we obtain ${{L_Y}^*}_{\mid \sM(Y)}=\tilde{I}^{\phi}$  by Lemma~\ref{l:LDP}; this proves the last assertion of part 1.
The hypothesis in  part 2   is equivalent to the one of  Theorem~\ref{t:LDP} (strictly speaking, of its analogue given by Remark~\ref{remark-extension-general-systems})  by taking $\varphi=\phi_{\mid X}$. Consequently, we have  for  each open subset $\sG'$ of $\sM(X)$,
 \[\liminf_{t\rightarrow+\infty}\frac{1}{t}\log{{\average_t}}_{\mid X}[m_X](\sG')=\liminf_{t\rightarrow+\infty}\frac{1}{t}\log m\{x\in X:\average_t(x)\in\sG'\}\]
 \[\ge -\inf_{\sG'} I^{\phi_{\mid X}}=-\inf_{\sG'}\tilde{I}^\phi,\] which proves  the assertion of part 2 concerning~$\sM(X)$.
The assertion concerning~$\sM(Y)$ follows by noting that
\begin{multline*}
m \{x \in Y:\average_t(x)\in\sG\}
\\ =
m\{x\in X:\average_t(x)\in\sG\cap\sM(X)\}
+ m \{x \in Y \setminus X : \average_t(x)\in\sG\}
\end{multline*}
for all open subsets $\sG$ of $\sM(Y)$, and using the above lower bounds.
\end{proof}

\begin{rema}\label{remark-strengthened-Kifer's-theo}
We explain here what  improvements    Theorem~\ref{Kifer-theo} brings with respect to
  the original version of  \cite[Theorem 3.4]{Kif90}.
  \begin{itemize}
  \item The latter treats the case where $P^\tau(\phi_{\mid X})=0$; this follows from the relation   \[P^{\tau}(\phi_{\mid X})=-\lim_{\delta\rightarrow 0}\lim_{t\rightarrow +\infty}\frac{1}{t}\log a_{\delta,t}\]as shows (\ref{Kifer-theo-eq4.1}), and   the general assumption there    which requires that for each~$\delta > 0$,
  \begin{equation}\label{remark-strengthened-Kifer's-theo-eq0}
\lim_{t\rightarrow+\infty} \frac{1}{t}\log a_{\delta,t}=0.
  \end{equation}
   \item  We do not require  that  $\textnormal{supp\ }m=Y$; in fact,   we only need that $m(X)>0$  in order to have the lower bounds  in $\sM(X)$, and that is ensured by the hypotheses.
  \item The hypothesis in part 1 of Theorem~\ref{Kifer-theo} in order to have the upper bounds in $\sM(Y)$ is   weaker than the one of \cite[Theorem 3.4]{Kif90}, where it is required  that (ii) holds for all $t\in\mathfrak{T}$, all $x\in Y_t$ and all $\delta>0$, and that moreover~\eqref{remark-strengthened-Kifer's-theo-eq0} holds for all~$\delta > 0$.
   \item The hypothesis in part 2 to get the lower bounds  in $\sM(X)$ is weaker than the one of \cite[Theorem 3.4]{Kif90} since this latter requires the existence of a dense vector subspace of $C(Y)$; furthermore, these bounds are stronger than the ones in $\sM(Y)$.
 \end{itemize}
\end{rema}

\begin{rema}\label{remark-extension-general-systems}
For each integer $d\ge 1$ we put  $\mathbb{Z}^{d}_+=\{x\in\mathbb{Z}^d: x_i\ge 0, 1\le i\le d\}$,  and let $\tau$ be a continuous
 representation  of the semi-group $\mathfrak{T}\in\{
\mathbb{Z}^{d}_+,\mathbb{R}_+\}$ (resp. group $\mathfrak{T}=\mathbb{Z}^{d}$) in the semi-group of continuous endomorphisms
 (resp. group of homeomorphisms) of  $X$, let
 $\sM^{\tau}(X)$,
 $h^{\tau}_{\cdot}$,
$P^{\tau}(\cdot)$ be  the  obvious analogues of
  $\sM(X,T)$,
$h_{\cdot}(T)$, $P(T,\cdot)$, respectively,   and assume that
 $h^{\tau}$ is finite and upper semi-continuous (when $\mathfrak{T}$ is continuous,   $h^{\tau}_{\cdot}$ and $P^{\tau}(\cdot)$ are taken as  the entropy and pressure of  the time-one map, respectively). Let $(\Omega_\alpha)_{\alpha\in\wp}$ be a net of Borel probability measures on $\sM(X)$ (in place of $( \Omega_n)_{n \ge 1}$), and  let $(t_\alpha)_{\alpha\in\wp}$ be a net in $(0,+\infty)$ converging to $0$ (in place of $(1/n)_{n \ge 1}$). It is then straightforward to verify that the statement  as well as  the proof  of Theorem~\ref{t:LDP}  work verbatim with the above changes (although the proof refers  to some  results of \cite{DemZei98} which are stated  for nets indexed by positive reals, these results remain valid for general nets).
 Indeed, Lemma~\ref{l:LDP} remains true by the variational principle relating $P^\tau$ y $h^\tau$,    the others  required ingredients   are given by  the  functional equality (\ref{e:functional equality}) and the hypothesis on $\cW$, so that we just have to change the symbols in the proof.
\end{rema}

\begin{rema}\label{remark-Kifer's-theo-invariant-case}
When $Y$ is $F^t$-invariant for all $t\in \mathfrak{T}$, then $Y=X$ and the proof of Theorem~\ref{Kifer-theo} reveals  that condition (ii)  ensures that the equality (\ref{e:functional equality}) of   Theorem~\ref{t:LDP} holds (more exactly, of its extension  given by Remark~\ref{remark-extension-general-systems}); the second hypothesis of part 2 of Theorem~\ref{Kifer-theo}  is  equivalent to the hypothesis on $\cW$ of  Theorem~\ref{t:LDP}.
 Consequently, all the conclusions of Theorem~\ref{Kifer-theo} follows from the general version of Theorem~\ref{t:LDP} given by Remark~\ref{remark-extension-general-systems}.
\end{rema}

\bibliographystyle{alpha}

\end{document}